\documentclass[a4paper,10pt]{amsart}
\usepackage{amsmath,amsfonts,amssymb, amsthm, xypic}
\usepackage{epsfig,psfrag}
\usepackage[all]{xy}

\def\ind{{\lim\limits_{\to}}}

\def\ss{\textbf{ss}}

\def\ben{{\mathfrak{b}}}
\def\gen{{\mathfrak{g}}}
\def\vvec{{\textbf{vec}}}
\def\hen{{\mathfrak{h}}}
\def\pen{{\mathfrak{p}}}

\def\nen{{\mathfrak{n}}}

\def\Ec{{\mathcal E}}
\def\Oc{{\mathcal O}}

\def\dim{\mathrm{dim}}

\def\End{\mathrm{End}}
\def\ad{{\mathrm{ad}}}

\def\Z{\mathbb{Z}}

\def\E{{X}}

\def\H{\mathbf{H}}
\def\Cb{\mathbf{C}}
\def\Ab{\mathbf{A}}
\def\Vb{\mathbf{C}'}

\def\D{\mathbf{D}}

\def\N{\mathbb{N}}

\def\lto{\longrightarrow}
\def\a{\alpha}
\def\b{\beta}
\def\qlb{\overline{\mathbb{Q}}_l}

\def\CC{\mathbb{C}}
\def\UU{\mathbf{U}}
\def\ZZ{\mathbb{Z}}
\def\QQ{\mathbb{Q}}

\def\x{\mathbf{x}}

\def\fqb{\overline{\mathbb{F}_q}}
\def\fq{\mathbb{F}_q}

\def\mod{{\text{mod}}}
\def\on{{\ \text{on}\ }}

\def\Sum{{\rm{\sum}}}

\newtheorem{theo}{\bf{Theorem}}[section]
\newtheorem{lem}[theo]{Lemma}

\newtheorem{cor}[theo]{Corollary}
\newtheorem*{coro}{Corollary}
\newtheorem{prop}[theo]{Proposition}

\newtheorem{conj}[theo]{Conjecture}

\newtheorem{theore}{\bf{Theorem}}

\numberwithin{equation}{section}

\title{Hall algebras of curves, commuting varieties and Langlands duality}
\author{O. Schiffmann, E. Vasserot}

\begin{document}

\begin{abstract} We construct an isomorphism between the (universal) spherical Hall algebra of a smooth projective curve
of genus $g$ and a convolution algebra in the (equivariant) 
$K$-theory of the genus $g$ commuting varieties
$C_{\mathfrak{gl}_r}=\big\{ (x_i, y_i) \in 
\mathfrak{gl}_r^{2g}\;;\; \sum_{i=1}^g [x_i,y_i]=0\big\}$. 
We can view this isomorphism as a version 
of the geometric Langlands duality in the formal neighborhood of 
the trivial local system, for the group $GL_r$. 
We extend this to all reductive groups and we compute the image, 
under our correspondence, of the skyscraper sheaf 
supported on the trivial local system.

\end{abstract}

\maketitle

\tableofcontents

\section{Introduction}

\vspace{.1in}

\paragraph{\textbf{0.1}} Let $X$ be a smooth connected projective curve of genus
$g$ defined over a finite field $\mathbb{F}_q$. Let $Bun_{r}X$ stand for the 
set of all (isomorphism classes of) vector bundles of rank $r$ over $X$. 
Consider the vector space 
$$\H_{Vec(X)}=\bigoplus_{r \geq 1} Fun(Bun_{r}X, \CC),$$ where 
$Fun( \cdot, \CC)$ denotes the set of complex valued functions with finite 
support. Let $P \subset GL_r$ be a parabolic subgroup with Levi factor 
$L \simeq GL_s \times GL_t$. The convolution diagram
\begin{equation}\label{E:intro1}
\xymatrix{ & Bun_PX \ar[dr]^-{p} \ar[ld]_-{q} &\\
Bun_sX \times Bun_tX & & Bun_{r}X}
\end{equation}
equips $\H_{Vec(X)}$ with an associative product and a coassociative coproduct.
Now let $Coh_0X$ stand for the set of all torsion coherent sheaves on $X$. 
A convolution diagram similar to (\ref{E:intro1}) allows one to equip 
$H_{Tor(X)}=Fun(Coh_0X, \CC)$ with the structure of an algebra 
(actually a Hopf algebra). This algebra acts on $\H_{Vec(X)}$. 
This construction has a natural interpretation in the theory of automorphic 
forms over function fields : $\H_{Vec(X)}$ is the direct sum (over $r$) of the 
spaces of unramified automorphic forms for the groups $GL_r$ over the ring of 
ad\`eles of $X$; the product and coproduct in $\H_{Vec(X)}$ correspond to 
parabolic induction and constant term maps respectively; 
finally, $H_{Tor(X)}$ is the algebra of (unramified) Hecke operators acting on 
the above automorphic forms.

In \cite{Kap}, Kapranov initiated the systematic study of $\H_{Vec(X)}$, using 
the language of Hall algebras. The algebra $\H_{Vec(X)}$ is 
generated as by
all cuspidal functions; Kapranov translated the functional equations 
satisfied by Eisenstein series associated to such cuspidal functions 
(or pairs of cuspidal functions) into commutation relations between the 
corresponding generators. These commutation relations bear a strong 
ressemblance with those appearing in Drinfeld's new realization of quantum 
affine algebras. In fact, Kapranov fully determined $\H_{Vec(X)}$ when 
$X=\mathbb{P}^1$ and identified it with the quantum group 
$\mathbf{U}^+_{v}(\widehat{\mathfrak{sl}}_2)$, where $v=q^{-1/2}$. 
The algebra $\H_{Vec(X)}$ is also explicitly described when $X$ is an 
elliptic curve, see \cite{BS}, \cite{Fratila}.
In this case it is related to double affine Hecke algebras, see \cite{SV1}.

\vspace{.2in}

\paragraph{\textbf{0.2}} In this paper, we let $X$ be a curve of arbitrary genus, but we restrict our attention to a subalgebra of $\H_{Vec(X)}$ generated by cuspidal functions \textit{of rank one}. More precisely, we define the \textit{spherical Hall algebra} $\mathbf{U}^>_X$ to be the subalgebra of $\H_{Vec(X)}$ generated by the characteristic functions $1_{Pic^dX}$ for $d \in \Z$. Our motivation here is to understand the
structure of this algebra, as well as its place in the Langlands program.

Our first result gives a combinatorial realization of $\UU^>_X$ in terms of \textit{shuffle algebras}. Let $\alpha_1, \overline{\alpha_1}, \ldots, \alpha_g, \overline{\alpha_g}$ be the Weil numbers of $X$ and set 
$$g_X(z)=z^{g-1}
\frac{1-qz}{1-z^{-1}}\prod_{i=1}^g (1-\a_iz^{-1})(1-\bar\a_iz^{-1}).$$
Let $\CC(x_1, \ldots, x_r)^{ \mathfrak{S}_r}$ stand for the space of symmetric rational functions in $r$ variables. We equip the graded space
$\mathbf{V}=\CC 1 \oplus \bigoplus_{r \geq 1} \CC(x_1, \ldots, x_r)^{ \mathfrak{S}_r}$ with the shuffle product
\begin{equation}\label{E:intro2}
P(x_1, \ldots, x_r) \star Q(x_1, \ldots, x_{s})=\sum_{w \in
Sh_{r,s}} w  \bigg( \hspace{-.1in}\prod_{\substack{1 \leqslant i
\leqslant r\\ r+1 \leqslant j \leqslant r+s}}
\hspace{-.15in}g_X(x_i/x_j) P(x_1, \ldots, x_r) Q(x_{r+1},
\ldots, x_{r+s})\bigg)
\end{equation}
where $Sh_{r,s} \subset \mathfrak{S}_{r+s}$ stands for the set of all $(r,s)$ shuffles on $r+s$ letters. Let $\mathbf{A}_{g_X(z)}$ be the subalgebra of $\mathbf{V}$
generated by $\CC[x_1^{\pm 1}] \subset \mathbf{V}_1$. It is a graded subalgebra
$\mathbf{A}_{g_X(z)} =\bigoplus_{ r \geq 0} \mathbf{A}_r$. Shuffle algebras of 
the above kind associated to rational functions have been studied in particular
by B.~Feigin and his collaborators, see e.g. \cite{FO}.

\vspace{.1in}

\begin{theore} The assignement ${1}_{Pic^lX} \mapsto x^l_1$ 
extends to an algebra isomorphism
$$\Upsilon_{\E}:\UU^>_{\E}\stackrel{\sim}{\longrightarrow}\Ab_{g_X(z)}.$$
\end{theore}

\vspace{.1in}

The map $\Upsilon_X$ is essentially the constant term map (restriction to the torus), and Theorem 1 is a consequence of the Gindikin-Karpelevich formula. The fact that the constant term map lands in a space of \textit{symmetric} functions is a manifestation of the functional equation for Eisenstein series. The Hecke operators preserve the spherical Hall algebra $\UU^>_X$, and their action on each graded component
$\UU^>_X[r]$ is directly related to the action of $\CC[x_1^{\pm 1}, \ldots, x_r^{\pm 1}]^{\mathfrak{S}_r} \simeq Rep\;GL_r$ on $\Ab_r$. It would therefore be interesting to determine precisely the structure of $\mathbf{A}_r$ as a $\CC[x_1^{\pm 1}, \ldots, x_r^{\pm 1}]^{\mathfrak{S}_r}$-module; a first step is taken in Proposition~\ref{T:wheels} where the support of $\mathbf{A}_r$ is determined in terms of \textit{wheel conditions}.

A consequence of Theorem~1 is that $\UU^>_X$ only depends on the genus $g$ of $X$ and on its set of Weil numbers; this allows us to define , for each genus $g$, a `universal'
form $\UU^>_{R_a}$ of $\UU^>_X$ over the representation ring $R_a=Rep\;T_a$ of the torus
$$T_a=\{(\eta_1, \overline{\eta}_1, \ldots, \eta_g, \overline{\eta}_g) \in (\CC^{\times})^{2g}\;;\; \eta_i\overline{\eta}_i=\eta_j\overline{\eta}_j, \;\forall\;i,j\} \simeq (\CC^\times)^{g+1}.$$

\vspace{.2in}

\paragraph{\textbf{0.3}} The Langlands correspondence for function fields, proved for $GL_r$ by Drinfeld ($r=2$) and then by L. Lafforgue (for all $r$) sets up a bijection between cuspidal Hecke eigenforms of rank $r$ over $X$ and irreducible $n$-dimensional $l$-adic representations of the Galois group $Gal(\overline{F}/F)$ where $F$ is the function field of $X$. The only cuspidal Hecke eigenform belonging to the spherical Hall algebra $\UU^>_X$ is the constant function $1_{PicX}$, which corresponds to the trivial (one-dimensional) Galois representation. Hence it is natural to expect that the Langlands correspondence relevant to $\UU^>_X$ will involve Galois representations 'close' to the trivial one.

At the moment we are not able to understand the role of the algebra $\UU^>_X$ in this Langlands picture for any \textit{fixed} curve $X$. However, we will give such an interpretation for the \textit{universal} spherical Hall algebra $\UU^>_{R_a}$ of a fixed genus $g$. For this, one has to pass to the complex setting and consider instead Beilinson and Drinfeld's version of the geometric Langlands program. Recall that this program aims at setting up an equivalence of (unbounded, triangulated) categories 
\begin{equation}\label{E:intro3}
Coh(\underline{Loc}_rX_\CC) \simeq D\hspace{-.03in}-\hspace{-.03in}mod(\underline{Bun}_rX_\CC)
\end{equation}
where $\underline{Loc}_rX_\CC$ and $\underline{Bun}_rX_\CC$ are the stacks of 
rank $r$ local systems and rank $r$ vector bundles on a complex curve $X_\CC$. 
When $r=1$ such an equivalence is known thanks to the work of Lang, Laumon and 
Rosenlicht (the \textit{geometric class field theory}, 
see e.g. \cite{LaumonICM}).

Let us consider the formal neighborhood $\widehat{triv}_r$ of the trivial local
system $triv_r$ in $\underline{Loc}_rX_\CC$ and let us assume for simplicity 
that $g >0$. By deformation theory $\widehat{triv}_r$ is isomorphic to the 
formal neighborhood of $\{0\}$ in the (singular) quotient stack
$C_{\mathfrak{gl}_r}/GL_r$, where $C_{\mathfrak{gl}_r}$ is the quadratic 
cone of solutions to the Maurer-Cartan equation
\begin{equation}\label{E:intro4}
C_{\mathfrak{gl}_r}=\{u \in H^1(X_\CC,\mathfrak{gl}_r)\;;\; [u,u]=0\}.
\end{equation}
Here $[\,,\,]$ is the Lie bracket on the dg Lie algebra 
$H^{\bullet +1}(X_\CC, \mathfrak{gl}_r)$ given by
$[c \otimes x, c' \otimes y]=c \cdot c' \otimes [x,y]$. 
Fixing a symplectic basis of $H^1(X_\CC,\CC)$ yields an identification
\begin{equation}\label{E:intro5}
C_{\mathfrak{gl}_r}=\big\{(a_i, b_i)_{i=1, \ldots, g} \in (\mathfrak{gl}_r)^{2g}\;;\; \sum_i [a_i,b_i]=0\big\}.
\end{equation}
The symplectic group $Sp(H^1(X_\CC,\CC), \cdot)$ naturally acts on $C_{\mathfrak{gl}_r}/GL_r$, and so does the torus
$$T_s=\{(e_1, f_1, \ldots, e_g, f_g) \in 
(\CC^{\times})^{2g}\;;\; e_if_i=e_jf_j, \;\forall\;i,j\} 
\simeq (\CC^\times)^{g+1}.$$
Let $P$ be a parabolic subgroup of $GL_r$ with Levi factor $GL_s \times GL_t$. Using the diagram
\begin{equation}\label{E:intro6}
\xymatrix{ & \underline{Loc}_PX_\CC \ar[dr]^-{p} \ar[ld]_-{q} &\\
\underline{Loc}_sX_\CC \times \underline{Loc}_tX_\CC & & \underline{Loc}_{r}X_\CC}
\end{equation}
(and its restriction to formal neighborhoods of trivial local systems)
we define a convolution product in equivariant $K$-theory
$$K^{T_s}(C_{\mathfrak{gl}_s}/GL_s) \otimes K^{T_s}(C_{\mathfrak{gl}_t}/GL_t) 
\to K^{T_s}(C_{\mathfrak{gl}_r}/GL_r).$$
This equips the graded space 
$$\mathbf{C}=\CC 1 \oplus \bigoplus_{r \geq 1} 
K^{T_s}(C_{\mathfrak{gl}_r}/GL_r)$$ with the structure of an associative 
algebra over the representation ring $R_s=Rep\;T_s$. Observe that the 
formal neighborhood $\widehat{triv}_r$, and hence
the algebra $\mathbf{C}$, only depend, up to isomorphism, on the genus $g$ 
of $X_\CC$.

Let $\mathbf{C}'$ be the subalgebra of $\mathbf{C}$ generated by 
$\mathbf{C}_1=K^{T_s \times GL_1}(\CC^{2g})=R_s[x^{\pm 1}]$. 
Inspired by the note \cite{Groj}, in the context of quivers, we prove the 
following result. Put 
$$k(z)=\frac{1-p^{-1}z^{-1}}{1-z}\prod_{i=1}^g (1-e_i^{-1}z)(1-f_i^{-1}z) \in R_s(z)$$
where we have set $p=e_if_i$ (for any $i$). Let $\mathbf{A}_{k(z)}$ be the $R_s$-shuffle algebra associated to $k(z)$.

\vspace{.1in}

\begin{theore} The assignment $x^l \mapsto x_1^l$ in degree one extends to an algebra isomorphism
$$\Phi: \mathbf{C}'/torsion \stackrel{\sim}{\longrightarrow} \mathbf{A}_{k(z)}.$$
\end{theore}

\vspace{.1in}

We conjecture that $\mathbf{C}'$ is a torsion free module. 

\vspace{.2in}

\paragraph{\textbf{0.4}} The functions $g(z)$ and $k(z)$ are of the same form. 
Combining Theorems 1 and 2 we deduce the following Langlands type isomorphism. 
We identify $R_a$ and $R_s$ via
$\eta_i \mapsto e_i^{-1}, \overline{\eta}_i \mapsto f_i^{-1}$.

\vspace{.1in}

\begin{coro} There exists a unique algebra (anti)isomorphism
\begin{equation}\label{E:intro7}
\Theta_R~: \mathbf{C}'/torsion \stackrel{\sim}{\longrightarrow}\dot{\UU}^>_{R_a}
\end{equation}
which restricts to geometric class field theory in degree one.
\end{coro}

\vspace{.1in}

In other words, the (universal) spherical Hall algebra is in correspondence 
with a certain convolution algebra of coherent sheaves on
the stacks $\underline{Loc}_rX_\CC$, supported in the formal neighborhood of 
the trivial local systems. The above corollary also suggests that
the Langlands correspondence should be compatible with the parabolic induction 
functors constructed from diagrams (\ref{E:intro1}) and
(\ref{E:intro6}) respectively. Note that we need to use a slightly twisted 
version $\dot{\UU}^>_{R_a}$ of $\UU^>_{R_a}$, see Section~1.13. This twist
is indeed very classical in the theory of Eisenstein series. Finally, one may 
show that the isomorphism $\Theta_R$ intertwines
the actions of the Wilson and Hecke operators, as required of (geometric) 
Langlands duality. 

Several remarks are in order at this point. First of all, the $K$-theoretic 
Hall algebra $\mathbf{C}'$ is constructed out of the moduli
stacks $\underline{Loc}_rX_\CC$ for a \textit{complex} curve $X_\CC$ of genus 
$g$ while the spherical Hall algebra $\UU^>_{R_a}$
is defined using curves of genus $g$ over \textit{finite} fields. Moreover, 
$\mathbf{C}'$ is a Grothendieck group of $T_s$-\textit{equivariant} coherent 
sheaves while the structure of $\UU^>_{R_a}$ as an $R_a$-module comes from the 
Weil numbers of curves $X$ over $\mathbb{F}_q$.
In order to correct this discrepancy, and to get something closer in spirit 
to (\ref{E:intro3}) one is tempted to consider instead of $\UU^>_{R_a}$ the 
monoidal category $\mathcal{E}is_{X_\CC}$ of Eisenstein automorphic (perverse) 
sheaves over the stacks $\underline{Bun}_rX_{\CC}$,
see \cite{Laumon}, \cite{SLectures}. We expect the symplectic group 
$Sp(H^1(X_{\CC},\CC),\cdot)$ 
and the torus $T_a$ to act on $\mathcal{E}is_{X_{\CC}}$, and we 
expect the resulting $T_a$-graded Grothendieck group
to be isomorphic to (a completion of) $\UU^>_{R_a}$. When $\CC$ is replaced by $\overline{\mathbb{F}_q}$ the $T_a$-grading on $\mathcal{E}is_{X_{\fqb}}$ should come from the Frobenius weight structure of the 
Eisenstein sheaves, see \cite{SLectures} and \cite{Scano} for $g=1$. 
We hope to get back to this in the future.

We give two immediate applications of our Langlands isomorphism 
(\ref{E:intro7}). We determine the image, under $\Theta_R$, of the class 
$[\mathcal{O}_{triv_r}]$ of the skyscraper sheaf supported at the trivial local
system, see Proposition \ref{P:cohtriv}. 
Conversely, we determine the inverse image 
under $\Theta_R$ of the constant function $1_{Bun_rX}$, 
see Proposition \ref{P:buntriv}.

\vspace{.2in}

\paragraph{\textbf{0.5}} To wrap up this paper, we provide two natural 
extensions of our results. By construction the spherical Hall algebra 
$\UU_X$ is generated by the locally constant functions on the Picard group 
$Pic\;X$, i.e., by the simplest cuspidal functions in rank one. We define 
the \textit{principal Hall algebra} to be the subalgebra of $\H_{Vec(X)}$ 
generated by \textit{all} functions on $Pic\;X$. We generalize the description 
given in Theorem 1 as a shuffle algebra to the principal Hall algebra. Note 
that unlike the spherical Hall algebra, the principal Hall algebra depends on 
more than just the Weil numbers of $X$ : it also depends on the group 
structure of $Pic\;X$. 
Finally, the spherical Hall algebra and the K-theoretic 
Hall algebra both admit natural extensions to arbitrary reductive groups. 
And in fact neither the statement nor the proof of the Langlands isomorphism 
(\ref{E:intro7}) use any special feature of the groups $GL_r$. We sketch such 
a generalization here.

\vspace{.1in}

\paragraph{\textbf{0.6}} The content of the paper is as follows. In Section 1 
we recall the definition of the Hall algebra of $X$, and give the shuffle 
presentation of the spherical
Hall algebra $\UU_X$. Section 2 is devoted to the Hall algebra in the 
(equivariant) K-theory of the genus $g$ commuting varieties 
$C_{\mathfrak{gl}_r}$. In Section 3 we state and prove our Langlands 
isomorphism and give the two applications mentioned in \textbf{0.4.} above. 
The extensions of our results to the principal Hall algebra and to other 
reductive groups are sketched in two appendices at the end of the paper.

\vspace{.2in}

\section{Hall algebras of curves}

\vspace{.1in}

\paragraph{\textbf{1.1.}} Let $\E$ be a smooth connected projective curve of genus $g$ defined over some finite field $\fq$. Let $\zeta_{\E}(t) \in \CC(t)$ be its zeta function, i.e.,
$$\zeta_{\E}(t)=\exp \bigg( \sum_{d \geqslant 1} \# \E(\mathbb{F}_{q^d}) \frac{t^d}{d} \bigg).$$
It is well-known that
$$\zeta_\E(t)=\frac{\prod_{i=1}^g (1-\a_i t)(1-\bar\a_it)}{(1-t)(1-qt)},$$
where $\{\a_1, \bar\a_1, \ldots, \a_g, \bar\a_g\}$ are the Frobenius eigenvalues in
$H^1(\E_{\overline{\fq}}, \qlb)$. Recall that $|\a_i|=q^{1/2},$
i.e., $\a_i \bar\a_i=q$, for all $i=1, \ldots, g$.

\vspace{.2in}

\paragraph{\textbf{1.2.}} Let $Coh(\E)$ be the category of coherent sheaves over $\E$,
and let $Tor(\E)$ be the full subcategory of $Coh(\E)$ consisting of
torsion sheaves. We will be interested in the Hall algebra $\H_{\E}$
of $Coh(\E)$. We briefly recall its definition here for the reader's
convenience, and refer to \cite{SLectures}, especially Section~4.11
for its standard properties. Let $K'_0(\E)=\Z^2$ be the numerical
Grothendieck group of $Coh(\E)$. The class of a sheaf $\mathcal{F}$
is $(r_{\mathcal{F}},d_{\mathcal{F}})$ where we have abbreviated
$r_{\mathcal{F}}=rank(\mathcal{F})$,
$d_{\mathcal{F}}=deg(\mathcal{F})$. The Euler form on $K_0'(\E)$ is
given by
\begin{equation*}
\langle \mathcal{F}, \mathcal{G}\rangle=
\text{dim}\;\text{Hom}(\mathcal{F},\mathcal{G})-\text{dim}\;\text{Ext}^1(\mathcal{F},\mathcal{G})=(
1-g)r_{\mathcal{F}}r_{\mathcal{G}} + \left| \begin{array}{cc} r_{\mathcal{F}} & r_{\mathcal{G}} \\
d_{\mathcal{F}} & d_{\mathcal{G}} \end{array} \right|.
\end{equation*}
Let $\mathcal{I}$ be the set of isomorphism classes of coherent sheaves over $\E$.
Let us choose a square root $v$ of $q^{-1}$. As a vector space we have
$$\H_{\E}=\{f~: \mathcal{I} \to \CC; supp(f)\;\text{finite}\}=
\bigoplus_{\mathcal{F} \in \mathcal{I}} \CC 1_{\mathcal{F}}$$
where $1_{\mathcal{F}}$ denotes the characteristic function of $\mathcal{F} \in \mathcal{I}$.
The multiplication is defined as
$$(f \cdot g) (\mathcal{R})=\sum_{\mathcal{N} \subseteq
\mathcal{R}} v^{-\langle \mathcal{R}/\mathcal{N},\mathcal{N}\rangle} 
f(\mathcal{R}/\mathcal{N}) g(\mathcal{N})$$
and the comultiplication is
$$\Delta(f) (\mathcal{M},\mathcal{N})=\frac{v^{\langle \mathcal{M},\mathcal{N}\rangle}}
{|\text{Ext}^1(\mathcal{M},\mathcal{N})|}\sum_{\xi \in
\text{Ext}^1(\mathcal{M},\mathcal{N})} f(\mathcal{X}_{\xi})$$ where
$\mathcal{X}_{\xi}$ is the extension of $\mathcal{N}$ by
$\mathcal{M}$ corresponding to $\xi$. 
Note that the coproduct takes values in a completion $\H_{\E}
\widehat{\otimes} \H_{\E}$ of the tensor space  $\H_{\E}
\otimes \H_{\E}$ only, see e.g., (\ref{E:identishuffle}). 
The Hall algebra $\H_{\E}$ is $\Z^2$-graded, by the class in the Grothendieck group.
We will sometimes write $\Delta_{\a,\beta}$  in order to specify the
graded components of the coproduct.

\vspace{.2in}

\paragraph{\textbf{1.3.}} The triple $(\H_{\E},\cdot, \Delta)$ is not a
(topological) bialgebra, but it becomes one if we suitably twist the
coproduct. For this we introduce an extra subalgebra
$\boldsymbol{\mathcal{K}}=\CC[\boldsymbol{\kappa}_{r,d}]$, with $(r,d)
\in \Z^2$, and we define an extended Hall algebra
$\widetilde{\H}_{\E} = \H_{\E} \otimes \boldsymbol{\mathcal{K}}$
with relations
$$\boldsymbol{\kappa}_{r,d}\, \boldsymbol{\kappa}_{s,l}=\boldsymbol{\kappa}_{r+s,d+l},
\quad \boldsymbol{\kappa}_{0,0}=1,\quad \boldsymbol{\kappa}_{r,d}
1_{\mathcal{M}}\, \boldsymbol{\kappa}_{r,d}^{-1} =
v^{-2r(1-g)r_{\mathcal{M}}} 1_{\mathcal{M}}.$$ The new coproduct is
given by the formulas
$$\widetilde{\Delta}(\boldsymbol{\kappa}_{r,d})=
\boldsymbol{\kappa}_{r,d} \otimes \boldsymbol{\kappa}_{r,d},$$
$$\widetilde{\Delta}(f)=\sum_{\mathcal{M},\mathcal{N}} \Delta(f)(\mathcal{M},\mathcal{N})
1_{\mathcal{M}}\boldsymbol{\kappa}_{r_{\mathcal{N}},d_{\mathcal{N}}}
\otimes 1_{\mathcal{N}}.$$ Then $(\widetilde{\H}_\E, \cdot,
\widetilde{\Delta})$ is a topological bialgebra. Finally, let
$$(\;,\;)_G~: \widetilde{\H}_{\E} \otimes \widetilde{\H}_{\E} \to
\CC$$ be Green's Hermitian scalar product defined by
$$(1_{\mathcal{M}} \boldsymbol{\kappa}_{r,d}, 1_{\mathcal{N}}\boldsymbol{\kappa}_{s,l})_G=
\frac{\delta_{\mathcal{M},\mathcal{N}}}{\#\text{Aut}(\mathcal{M})}v^{-2(1-g)rs}.$$
This scalar product satisfies the Hopf property, i.e., we have
$$(ab,c)_G=(a \otimes b, \widetilde{\Delta}(c))_G,\quad a,b,c \in
\widetilde{\H}_{\E}.$$ Observe that the restriction of $(\;,\;)_G$
to $\H_{\E}$ is nondegenerate.
The Hall algebras  $\widetilde{\H}_{\E}$ is also $\Z^2$-graded.
We will write $\widetilde{\Delta}_{\a,\beta}$ in order to specify the
graded components of the coproduct.

\vspace{.2in}

\paragraph{\textbf{1.4.}} The category $Tor(\E)$ is a Serre subcategory of $Coh(\E)$.
It therefore gives rise to a sub bialgebra $\widetilde{\H}_{Tor(\E)}$ of $\widetilde{\H}_{\E}$ defined by
$$\widetilde{\H}_{Tor(\E)} =\bigg( \bigoplus_{\mathcal{M} \in Tor(\E)}
\CC 1_{\mathcal{M}} \bigg) \otimes \boldsymbol{\mathcal{K}}.$$
The decomposition of $Tor(\E)$ over the points of $\E$ gives rise to a similar decomposition
at the level of Hall algebras
$$\widetilde{\H}_{Tor(\E)} =\bigg( \bigotimes_{x \in \E} \H_{Tor(\E)_x} \bigg) \otimes \boldsymbol{\mathcal{K}}$$
where $Tor(\E)_x$ is the subcategory of torsion sheaves supported at
$x \in \E$. It is well-known that $\H_{Tor(\E)_x}$ is commutative
and cocommutative : it is isomorphic to the \textit{classical}
Hall algebra. 
For $d \geqslant 1$ we set
$$\mathbf{1}_{0,d}=\sum_{\substack{\mathcal{M} \in Tor(\E) \\ deg(\mathcal{M})=d}} 1_{\mathcal{M}}$$
and we define elements $T_{0,d}, \theta_{0,d}$ of $\widetilde{\H}_{Tor(\E)}$ through the relations
$$1+ \sum_{d \geqslant 1} \mathbf{1}_{0,d} s^d=\exp \bigg( \sum_{d}
\frac{T_{0,d}}{[d]}s^d\bigg),\quad 1+ \sum_{d \geqslant 1} \theta_{0,d}
s^d=\exp \bigg( (v^{-1}-v)\sum_{d\geqslant 1} T_{0,d}s^d\bigg).$$ 
Here as
usual $[n]=(v^n-v^{-n})/(v-v^{-1})$. We set also $\mathbf{1}_{0,0}=T_{0,0}=\theta_{0,0}=1$.
The following lemma summarizes
the properties of the elements $T_{0,d}$, $\theta_{0,d}$ which we
will later need.

\vspace{.1in}

\begin{lem}\label{L:Hall1}
The following hold for all $d \in \N$ 

$(a)$ $\widetilde{\Delta}(T_{0,d})=T_{0,d} \otimes 1 +
\boldsymbol{\kappa}_{0,d} \otimes T_{0,d}$,

$(b)$ $\widetilde{\Delta}(\theta_{0,d})=\sum_{l=0}^d
\theta_{0,l}\boldsymbol{\kappa}_{0,d-l} \otimes \theta_{0,d-l}$,

$(c)$ $(T_{0,d},T_{0,d})_G=\big( v^{d-1}\# \E(\mathbb{F}_{q^d}) [d]
\big)/(d q-d).$
\end{lem}
\begin{proof} See \cite[Section~4.11]{SLectures}.
\end{proof}

\vspace{.1in}

\noindent The sets $\{\mathbf{1}_{0,d}\;;\; d \in \N\}$,
$\{T_{0,d}\;;\; d \in \N\}$, and $\{\theta_{0,d}\;;\; d \in \N\}$
all generate the same subalgebra $\UU^0_{\E}$ of
$\widetilde{\H}_{Tor(\E)}$. It is known that $$\UU^0_{\E}=
\CC[\mathbf{1}_{0,1}, \mathbf{1}_{0,2}, \ldots],$$ i.e., the
commuting elements $\mathbf{1}_{0,d}$ for $d \geqslant 1$ are
algebraically independent. Clearly, the same holds also for the collections of generators
$\{\theta_{0,d}\}$, $\{T_{0,d}\}$.

\vspace{.2in}

\paragraph{\textbf{1.5.}} Now let $Vec(\E)$ be the exact subcategory of $Coh(\E)$ consisting
of vector bundles. It gives rise to a subalgebra $\H_{Vec(\E)}$
of $\H_{\E}$. This subalgebra is, however, not stable under the
coproduct $\Delta$. The multiplication map yields isomorphisms
\begin{equation}\label{E:isomo}
\H_{Vec(\E)} \otimes \widetilde{\H}_{Tor(\E)}\to
\widetilde{\H}_{\E}, \qquad \H_{Vec(\E)} \otimes {\H}_{Tor(\E)}\to{\H}_{\E}.
\end{equation}
Indeed, if $\mathcal{V}$ is a vector bundle and $\mathcal{T}$ is a torsion sheaf then 
$$1_\mathcal{V} \cdot 1_{\mathcal{T}}=v^{-\langle \mathcal{V}, \mathcal{T}\rangle} 1_{\mathcal{V} \oplus \mathcal{T}},$$ because $\text{Ext}^1(\mathcal{V},\mathcal{T})=0$ and because the subsheaf $\mathcal{T} \subset \mathcal{V} \oplus \mathcal{T}$ is canonically defined. The comultiplication provides an inverse to (\ref{E:isomo}). To be more precise, given $\mathcal{V}$, $\mathcal{T}$ as above we have 
\begin{equation}\label{E:isomocop}
\Delta_{\overline{\mathcal{V}}, \overline{\mathcal{T}}}(1_{\mathcal{V}\oplus \mathcal{T}})=v^{\langle \mathcal{V},\mathcal{T}\rangle} 1_{\mathcal{V}} \otimes 1_{\mathcal{T}}.
\end{equation}
This comes from the fact that $\mathcal{T}$ is the only torsion subsheaf of $\mathcal{V}\oplus\mathcal{T}$ of  degree $d=deg(\mathcal{T})$.
We may use (\ref{E:isomo}) to define a projection
$\widetilde{\H}_{\E} \to \H_{Vec(\E)}$, which we will denote by
$\omega$. It satisfies
\begin{equation}\label{E:Hallpi}
\omega(u_{v}u_{t}\boldsymbol{\kappa}_{r,d})=\begin{cases} u_{v} & \text{if}\; u_{t}=1\\
0 &\text{otherwise}\end{cases},
\quad u_{v} \in \H_{Vec(\E)},\quad u_{t} \in \H_{Tor(\E)},\quad (r,d)
\in \Z^2.
\end{equation}

\vspace{.2in}

\paragraph{\textbf{1.6.}} There is a natural action of the Hall algebra $\widetilde{\H}_{Tor(\E)}$ on $\H_{Vec(\E)}$ by means of the so-called \textit{Hecke operators}, given by the formula
\begin{equation*}
\begin{split}
\widetilde{\H}_{Tor(\E)}\otimes \H_{Vec(\E)} \to \H_{Vec(\E)},
\quad u_0 \otimes u
\mapsto u_0 \bullet w= \omega( u_0 w),
\end{split}
\end{equation*}
in other words, we have $$(u_0 u'_0) \bullet v=u_0 \bullet ( u'_0 \bullet v).$$ 

\begin{prop}
\label{prop:Hecket0d}
The Hecke action on $\H_{Vec(\E)}$ preserves each graded component
$\H_{Vec(\E)}[r]$, and we have
$T_{0,d} \bullet w= [T_{0,d}, w]$ for $w \in \H_{Vec(\E)}$ and $d>0$.
\end{prop}

\begin{proof}
We have $T_{0,d} \bullet w = \omega (T_{0,d}w)=\omega( [T_{0,d},w])$. 
We claim that $[T_{0,d},w]$ already belongs to $\H_{Vec(\E)}$. Note that 
$$\widetilde{\Delta}(T_{0,d}w)=(T_{0,d} \otimes 1 + \boldsymbol{\kappa}_{0,d} \otimes T_{0,d}) \widetilde{\Delta}(w),$$
$$\widetilde{\Delta}(w T_{0,d})=\widetilde{\Delta}(w)(T_{0,d} \otimes 1 + \boldsymbol{\kappa}_{0,d} \otimes T_{0,d}),$$
$$\widetilde{\Delta}(w) \in \widetilde{\H}_{\E} \widehat{\otimes} \H_{Vec(\E)}.$$ 
Therefore $\widetilde{\Delta}([T_{0,d},w])=[\widetilde{\Delta}(T_{0,d}), \widetilde{\Delta}(w)]$ has no component in $\H_{Vec(\E)}\boldsymbol{\mathcal{K}} \otimes \H_{Tor(\E)}$.
Hence (\ref{E:isomocop}) implies that $[T_{0,d},w]$ belongs to $\H_{Vec(\E)}$.
\end{proof}

\vspace{.2in}

\paragraph{\textbf{1.7.}} For $d \in \Z$ let $Pic^d(\E)$ be the set of line bundles over $\E$ of degree
$d$. Put
$$\mathbf{1}^{\ss}_{1,d}=\mathbf{1}_{Pic^d(\E)}=\sum_{\mathcal{L} \in Pic^d(\E)} 1_{\mathcal{L}}.$$
Let us denote by $\UU^>_{\E}$ the subalgebra of
$\H_{\E}$ generated by $\{\mathbf{1}^{\ss}_{1,d}\;;\;
d \in \Z\}$. The \textit{spherical Hall algebra} of $\E$ is the 
subalgebra\footnote{this subalgebra is usually denoted $\UU^+_{\E}$,
e.g., in \cite{BS}, \cite{SV1}, \cite{SV2}. We simply write $\UU_{\E}$ here 
since we do not consider any Drinfeld double.} 
$\UU_{\E}$ of $\H_{\E}$ generated by $\UU_{\E}^0$ and 
$\UU^>_{\E}$. We also define
$\widetilde\UU_{\E}$ as the subalgebra of $\widetilde{\H}_{\E}$ generated by
$\boldsymbol{\mathcal{K}}$ and $\UU_{\E}$. It is known that $\widetilde{\UU}_{\E}$ is a
topological sub bialgebra of $\widetilde{\H}_{\E}$. 
Finally, set $\widetilde{\UU}^0_{\E} =  \UU^0_{\E}\boldsymbol{\mathcal{K}} $. The multiplication map gives an
isomorphism $\UU^>_{\E} \otimes \widetilde{\UU}^0_{\E}
\stackrel{\sim}{\lto} \widetilde\UU_{\E}$, see e.g.,
\cite[Section 6]{SV1}.
The coproduct of the elements $\mathbf{1}^{\ss}_{1,d}$ can be explicitly computed,
see e.g., \cite[Ex. 4.12]{SLectures},
\begin{equation}\label{E:Hall1}
\widetilde{\Delta}(\mathbf{1}^{\ss}_{1,d})=\mathbf{1}^{\ss}_{1,d} \otimes 1 +
\sum_{l \geqslant 0} \theta_{0,l} \boldsymbol{\kappa}_{1,d-l} \otimes
\mathbf{1}^{\ss}_{1,d-l}.
\end{equation}

\vspace{.2in}

\paragraph{\textbf{1.8.}} 
Our aim for the rest of this first section is to give a 
presentation of the algebra
$\UU_{\E}$. When $\E$ is of genus $0$ or $1$
this was done explicitly in \cite{Kap}, \cite{BS}. In higher genus,
we will provide a more implicit presentation, which will however
suffice for our purposes here. Our approach, which is based on the notion of 
shuffle algebras, see e.g., \cite{Rosso}, \cite{Feigin}, 
can be developed for the
Hall algebra of a more or less arbitrary hereditary category.
We begin with a couple of preliminary lemmas.

\vspace{.1in}

\begin{lem}\label{L:Hall2} We have
$$[T_{0,d},\mathbf{1}^{\ss}_{1,l}]=
\frac{v^d \#\E(\mathbb{F}_{q^d})[d]}{d}\mathbf{1}^{\ss}_{1,l+d},
\quad l\in\ZZ,\quad d>0.$$
\end{lem}
\begin{proof} The proof is very close to \cite[Lemma~4.12]{BS}, 
see also \cite[Theorem~6.3]{SV1}.
Because the left hand side is an element of
$\UU_{\E}$, we may write
$$[T_{0,d},\mathbf{1}^{\ss}_{1,l}]=\sum_n \mathbf{1}^{\ss}_{1,n} u_n,\quad
u_n \in \UU^0_{Tor(\E)}.$$
But because $[T_{0,d},\mathbf{1}^{\ss}_{1,l}]=T_{0,d}\bullet \mathbf{1}^{\ss}_{1,l} \in \H_{Vec(\E)}$ 
we must have in fact 
$$[T_{0,d},
\mathbf{1}^{\ss}_{1,l}]=c_{l,d} \mathbf{1}^{\ss}_{1,l+d},\quad c_{l,d} \in
\CC.$$ For this, we write
\begin{equation*}
\begin{split}
c_{l,d} (\mathbf{1}^{\ss}_{1,l+d},\mathbf{1}^{\ss}_{1,l+d})_G&=(\mathbf{1}^{\ss}_{1,l+d},[T_{0,d}, \mathbf{1}^{\ss}_{1,l}])_G\\
&=(\mathbf{1}^{\ss}_{1,l+d},T_{0,d}, \mathbf{1}^{\ss}_{1,l})_G\\
&=(\widetilde{\Delta}(\mathbf{1}^{\ss}_{1,l+d}), T_{0,d} \otimes \mathbf{1}^{\ss}_{1,l})_G\\
&=(\theta_{0,d}\boldsymbol{\kappa}_{1,l}, T_{0,d})_G
(\mathbf{1}^{\ss}_{1,l},\mathbf{1}^{\ss}_{1,l})_G
\end{split}
\end{equation*}
and hence $c_{l,d}=(\theta_{0,d},T_{0,d})_G$. Developing
$\theta_{0,d}$ in terms of the $T_{0,n}$'s, using the Hopf property
of $(\;,\;)_G$ and Lemma~\ref{L:Hall1} a) we obtain
$$(\theta_{0,d},T_{0,d})_G=(v^{-1}-v)(T_{0,d},T_{0,d})_G=\frac{v^d \#\E(\mathbb{F}_{q^d})[d]}{d}$$
as wanted.
\end{proof}

\vspace{.1in}

\begin{cor}\label{C:corred} Define complex numbers $\xi_k$, $k \geqslant 0$ 
by $\omega(\theta_{0,k}
\mathbf{1}^{\ss}_{1,l})=\xi_k \mathbf{1}^{\ss}_{1,l+k}$.  Then we have

$(a)$ for $l \in \Z$ and $d \geqslant 0$,
\begin{equation*}
\theta_{0,d} \mathbf{1}^{\ss}_{1,l}=\sum_{n=0}^d \xi_{n}
\mathbf{1}^{\ss}_{1,l+n}\theta_{0,d-n},
\end{equation*}

$(b)$ as a series in $\CC[[s]]$ we have
\begin{equation*}
\sum_{d \geqslant 0} \xi_d s^d=\frac{\zeta_{\E}(s)}{\zeta_{\E}(q^{-1}s)}.
\end{equation*}
\end{cor}
\begin{proof} Statement $(a)$ is a consequence of Lemma~\ref{L:Hall2} and the definition of $\theta_{0,d}$.
We prove $(b)$. By
Lemma~\ref{L:Hall2} we have $$\omega(T_{0,d} \mathbf{1}^{\ss}_{1,l})=c_d
\mathbf{1}^{\ss}_{1,l+d},\quad c_d=v^d \#\E(\mathbb{F}_{q^d})[d]/{d}.$$
It follows that $\omega( T_{0,d_1} \cdots
T_{0,d_r} \mathbf{1}^{\ss}_{1,l})=c_{d_1} \cdots
c_{d_r}\mathbf{1}^{\ss}_{1,l+d_1+\cdots+d_r}$. Now, we have
\begin{equation*}
\sum_{d \geqslant 0} \xi_d s^d= \exp \bigg( (v^{-1}-v)\sum_{d \geqslant 1} c_d
s^d\bigg) =\exp \bigg( \sum_{d \geqslant 1} \#\E(\mathbb{F}_{q^d})
(1-v^{2d})\frac{s^d}{d} \bigg)
=\frac{\zeta_{\E}(s)}{\zeta_{\E}(q^{-1}s)}.
\end{equation*}
\end{proof}

\vspace{.2in}

\paragraph{\textbf{1.9.}} 
In order to give presentations of ${\UU}_{\E}$ and $\UU^>_{\E}$
we introduce the so-called \textit{constant term map}. For $r \geqslant 1$ we set
$$J_r: \UU^>_{\E}[r] \to  \UU^>_{\E}[1] \widehat{\otimes} \cdots \widehat{\otimes} \UU^>_{\E}[1], \qquad u \mapsto (\omega \otimes \cdots \otimes \omega) \Delta_{1, \ldots, 1}(u)$$
and denote by $J~: \UU^>_{\E} \to \bigoplus_r \big(\UU^>_{\E}[1]\big)^{\widehat{\otimes} r}$ the sum of the maps $J_r$.   
Writing 
$$J(u)=\sum_{\mathcal{L}_1, \ldots, \mathcal{L}_r} u(\mathcal{L}_1, \ldots, \mathcal{L}_r) 1_{\mathcal{L}_1} \otimes \cdots \otimes 1_{\mathcal{L}_r},$$ 
we have
\begin{equation}\label{E:constant}
u(\mathcal{L}_1, \ldots, \mathcal{L}_r)=
\frac{1}{(q-1)^r}(J(u), 1_{\mathcal{L}_1} \otimes \cdots \otimes 1_{\mathcal{L}_r})_G = \frac{1}{(q-1)^r}( u, 1_{\mathcal{L}_1} \cdots 1_{\mathcal{L}_r})_G
\end{equation}
which coincides with the standard notion of constant term in the theory of automorphic forms (up to the factor $(q-1)^{-r}$), see \cite{Kap}.
Observe that because $J_r$ lands in $\big(\UU^>_{\E}[1]\big)^{\widehat{\otimes} r}$ and $\UU^>_{\E}[1]=\bigoplus_d \mathbf{1}^{\ss}_{1,d}$, the function $u(\mathcal{L}_1, \ldots, \mathcal{L}_r)$ only depends on the respective degrees $d_1, \ldots, d_r$ of the line bundles $\mathcal{L}_1, \ldots, \mathcal{L}_r$.

\vspace{.1in}

\begin{lem}\label{L:constantinj} The constant term map $J:\UU^>_{\E} \to \bigoplus_r \big(\UU^>_{\E}[1]\big)^{\widehat{\otimes} r}$ is injective.
\end{lem}
\begin{proof} Let $u \in \UU^>_{\E}[r]$ be in the kernel of $J_r$. Then by (\ref{E:constant}) we have $(u, 1_{\mathcal{L}_1} \cdots 1_{\mathcal{L}_r})_G=0$ for all $r$-tuples of line bundles $(\mathcal{L}_1, \ldots, \mathcal{L}_r)$. In particular, 
$$(u, \mathbf{1}^{\ss}_{1,d_1} \cdots \mathbf{1}^{\ss}_{1,d_r})_G=0,\quad
(d_1, \ldots, d_r) \in \Z^r.$$ By construction $\UU^>_{\E}$ is generated by the elements
$\mathbf{1}^{\ss}_{1,d}$ hence $(u, \UU^>_{\E})_G=0$. But the restriction of $(\;,\;)_G$ to $\UU^>_{\E}$ is known to be nondegenerate, see e.g., \cite[Thm. 4.52]{SLectures}. Therefore $u=0$ and $J_r$ is injective as wanted.\end{proof}

\vspace{.1in}

\noindent The objective is now to determine as precisely as possible the image 
of $J$ and to write the product and coproduct structure of $\UU^>_{\E}$
in terms of $\bigoplus_r\big(\UU^>_{\E}[1]\big)^{\widehat{\otimes} r}$. 
For this it will be convenient to identify $\UU^>_{\E}[1]$ with 
$\CC[x^{\pm 1}]$ via the assignment 
$$\mathbf{1}^{\ss}_{1,d} \mapsto x^d,\quad d\in\ZZ.$$ 
Thus we have 
\begin{equation}\label{E:identishuffle}
\begin{split}
&\big(\UU^>_{\E}[1]\big)^{\otimes r}\simeq\CC[x_1^{\pm 1},\ldots,x_r^{\pm 1}],
\cr
&\big(\UU^>_{\E}[1]\big)^{\widehat{\otimes} r} \simeq \CC\big[x_1^{\pm 1}, \ldots, x_r^{\pm 1}\big] \big[\big[ x_1/x_2, \ldots, x_{r-1}/x_r\big]\big].
\end{split}
\end{equation}
As it turns out, the (co)algebra structure of $\UU^>_{\E}$ may be extended to the whole of $\bigoplus_r\big(\UU^>_{\E}[1]\big)^{\widehat{\otimes} r},$ where it is nicely expressed as a shuffle algebra. Before writing the definition of these shuffles algebras, we begin with a few notations. 
Let $\mathfrak{S}_{r}$ be the group of
permutations of $\{1, \ldots, r\}.$ 
If $w\in \mathfrak{S}_{r}$ and $P(z_1, \ldots, z_r)$ a
function in $r$ variables then we set $wP(z_1, \ldots,
z_r)=P(z_{w(1)}, \ldots, z_{w(r)})$. Let
$$Sh_{r,s}=\{w \in \mathfrak{S}_{r+s}\;;\; w(i) < w(j)\;\text{if}\; 1 \leqslant i < j \leqslant r \;\text{or}\;
r < i<j\leqslant r+s\}$$ be the set of $(r,s)$-shuffles, i.e., the set of minimal lenght representatives
of the left cosets in $\mathfrak{S}_{r+s}/\mathfrak{S}_{r}\times\mathfrak{S}_{s}$. Write
$$I_{w}=\{(i,j)\;;\; 1 \leqslant i   < j \leqslant r+s,\; w^{-1}(i) > r \geqslant w^{-1}(j)\},\quad
w\in Sh_{r,s}.$$
Finally, let $h(z) \in \CC(z)$ be a fixed rational function. 
We define an associative algebra $\mathbf{F}_{h(z)}$ as follows.
As a vector space
$$\mathbf{F}_{h(z)}=\bigoplus_r \CC\big[x_1^{\pm 1}, \ldots, x_r^{\pm 1}\big] \big[\big[ x_1/x_2, \ldots, x_{r-1}/x_r\big]\big],$$ and the multiplication is given by
\begin{equation}\label{E:shuffle1}
P(x_1, \ldots, x_r) \star Q(x_1, \ldots, x_s)=
\sum_{\sigma \in Sh_{r,s}} \hspace{-.04in}
h(x_1,\dots,x_{r+s})\; 
\sigma \big(P(x_1, \ldots, x_r)
Q(x_{r+1}, \ldots, x_{r+s})\big).
\end{equation}
Here $$h_\sigma(x_1,\dots,x_{r+s})=\prod_{(i,j) \in I_{\sigma}}h(x_i/x_j),$$
and the rational function $h(x_i/x_j)$ is developed as a Laurent series in 
$x_1/x_2, \ldots, x_{r-1}/x_r$.
Note that the product 
(\ref{E:shuffle1}) is well-defined, i.e., it lands in the right completion, 
because the sum ranges over $(r,s)$-shuffles.
We equip also $\mathbf{F}_{h(z)}$ with a coproduct $\Delta~: \mathbf{F}_{h(z)} \to \mathbf{F}_{h(z)} \widehat{\otimes} \mathbf{F}_{(h(z)}$ defined by
\begin{equation}\label{E:coprodshuffle}
\Delta_{s,t}(x_1^{i_1} \cdots x_{r}^{i_r})=x_1^{i_1} \cdots x_s^{i_s} \otimes x_1^{i_{s+1}} \cdots x_{t}^{i_{s+t}},
\quad \Delta=\bigoplus_{r=s+t}\Delta_{s,t}.
\end{equation}
The subspace $\mathbf{F}_{h(z)}^{rat}$ of $\mathbf{F}_{h(z)}$ consisting of Laurent series which are expansions of rational functions forms a subalgebra, which is moreover stable under the coproduct.
We are ready to give a first description of $\UU^>_{\E}$ as a shuffle algebra. 
Recall that we have identified 
$\bigoplus_r\big(\UU^>_{\E}[1]\big)^{\widehat{\otimes} r}$ 
with the vector space $\mathbf{F}_{h(z)}$ via (\ref{E:identishuffle}).
 
 \vspace{.1in}
 
\begin{prop}\label{P:shuffle1} Set 
\begin{equation}\label{E:hx(z)}
h_{\E}(z)=q^{1-g}\zeta_{\E}(z)/\zeta_{\E}(q^{-1}z).
\end{equation}
The constant term map $J: \UU^>_{\E} \to \mathbf{F}_{h_{\E}(z)}$ is an algebra morphism such that 
$$(J_s \otimes J_t) \circ (\omega \otimes \omega)\circ\Delta_{s,t}=\Delta_{s,t} \circ J_r,\quad
r=s+t.$$
\end{prop}
\begin{proof} Because $\UU^>_{\E}$ is generated in degree one, to show that $J$ is a algebra homomorphism
it is enough to prove that for  $(d_1, \ldots,d_r) \in \Z^r$ we have
\begin{equation}\label{E:Hall11}
J_r(\mathbf{1}^{\ss}_{1,d_1} \cdots \mathbf{1}^{\ss}_{1,d_r})=J_1(\mathbf{1}^{\ss}_{1,d_1}) \star \cdots \star 
J_1(\mathbf{1}^{\ss}_{1,d_{r}})=x_1^{d_1} \star \cdots \star x_1^{d_r}.
\end{equation}
We will do this by computing the left hand side explicitly. We have 
\begin{equation}\label{E:Hall23}
\begin{split}
J_r(\mathbf{1}^{\ss}_{1,d_1} \cdots \mathbf{1}^{\ss}_{1,d_r})&=\omega^{\otimes r} \widetilde{\Delta}_{(1^r)}(\mathbf{1}^{\ss}_{1,d_1} \cdots \mathbf{1}^{\ss}_{1,d_r})\\
&=\omega^{\otimes r} \sum_{\sigma \in \mathfrak{S}_r} \widetilde{\Delta}_{\delta_{\sigma(1)}}(\mathbf{1}^{\ss}_{1,d_1}) \cdots
 \widetilde{\Delta}_{\delta_{\sigma(r)}}(\mathbf{1}^{\ss}_{1,d_r}),
 \end{split}
 \end{equation}
 where by definition $(\delta_1, \cdots, \delta_r)$ is the standard basis of $\Z^r$. 
In the above, we have made use of the fact that $\widetilde{\Delta}$ is a morphism of algebras. 
Set $\theta_d=\theta_{0,d}$ in an effort to unburden the notation. 
Using (\ref{E:Hall1}) we get
\begin{equation}\label{E:Hall3}
\widetilde{\Delta}_{\delta_k}(\mathbf{1}^{\ss}_{1,l})=
\sum_{d_1, \ldots, d_{k-1} \geqslant 0} \theta_{d_1}
\boldsymbol{\kappa}_{1,l-d_1} \otimes
\cdots \otimes \theta_{d_{k-1}} \boldsymbol{\kappa}_{1,l-\sum_{j <k}d_j} 
\otimes \mathbf{1}^{\ss}_{1,l-\sum_{j<k}d_j}\otimes 1 \otimes\cdots\otimes 1.
\end{equation} 
In order to compute the projection to vector bundles $\omega^{\otimes r}$, we need to reorder all the products appearing in the tensor components on the 
right hand side of (\ref{E:Hall23}) and to put them in the normal form 
$\UU^>_{\E} \cdot \widetilde{\UU}^0_{\E}$, see (\ref{E:Hallpi}). For this, 
note that Corollary~\ref{C:corred} yields
\begin{equation}\label{E:Hall234}
\begin{split}
\sum_{u \geqslant 0} \theta_{u} \boldsymbol{\kappa}_{1,p-u} \mathbf{1}^{\ss}_{1,k} \otimes \mathbf{1}^{\ss}_{1,l-u}&=q^{1-g}
\sum_{u \geqslant 0} \theta_{u} \mathbf{1}^{\ss}_{1,k} \boldsymbol{\kappa}_{p-u}  \otimes \mathbf{1}^{\ss}_{1,l-u}\\
&=q^{1-g}\sum_{u \geqslant 0}\sum_{v=0}^u \xi_v \mathbf{1}^{\ss}_{k+v} \theta_{u-v}\boldsymbol{\kappa}_{1,p-u} \otimes \mathbf{1}^{ss}_{1,l-u}\\
&=q^{1-g}\sum_{v \geqslant 0} \sum_{w \geqslant 0} \xi_v \mathbf{1}^{\ss}_{k+v} \theta_w \boldsymbol{\kappa}_{1,p-v-w} \otimes \mathbf{1}^{\ss}_{1,l-v-w}
\end{split}
\end{equation}
Let us introduce the automorphism $\gamma$ of $\UU^>_{\E}[1]$ defined by $\gamma(\mathbf{1}^{\ss}_{1,n})=\mathbf{1}^{\ss}_{1,n+1}$, and let us denote by $\gamma_i$ the operator $\gamma$ acting on the $i$th component of the tensor product. Using this and Corollary~\ref{C:corred} we may rewrite the right hand side of (\ref{E:Hall234}) as
\begin{equation}\label{E:Hall2345}
\begin{split}
q^{1-g}\sum_{v \geqslant 0} \xi_v \big( \gamma_1\gamma_2^{-1} 
&\boldsymbol{\kappa}_{0,-1}\big)^v \sum_{w}\mathbf{1}^{\ss}_{k+v} \theta_w \boldsymbol{\kappa}_{1,p-w} \otimes \mathbf{1}^{\ss}_{1,l-w}=\\
&= h_{\E}\big(\gamma_1\gamma_2^{-1}\boldsymbol{\kappa}_{0,-1}\big) \sum_{w}\mathbf{1}^{\ss}_{k+v} \theta_w \boldsymbol{\kappa}_{1,p-w} \otimes \mathbf{1}^{\ss}_{1,l-w}.
\end{split}
\end{equation}
For $\sigma \in \mathfrak{S}_r$ there is one contribution of the form (\ref{E:Hall2345}) in (\ref{E:Hall23}) for each inversion $(i,j)$ of $\sigma$. Note that $\boldsymbol{\kappa}_{0,1}$ is central and we have $\omega(u \boldsymbol{\kappa}_{0,1})=\omega(u)$ for all $u$ so that we may discard it. Thus all together we get
\begin{equation}\label{E:Hall34}
\begin{split}
J_r(\mathbf{1}^{\ss}_{1,d_1} \cdots \mathbf{1}^{\ss}_{1,d_r})
 &=\sum_{\sigma \in \mathfrak{S}_r}\prod_{(i,j) \in I_{\sigma}} h_{\E}\big(\gamma_i\gamma_j^{-1}\big) \, \sigma ( \mathbf{1}^{\ss}_{1,d_1} \otimes  \cdots \otimes \mathbf{1}^{\ss}_{1,d_r}).
 \end{split}
 \end{equation}
Observe that after the identification (\ref{E:identishuffle}) with Laurent series, the operator $\gamma_i$ simply becomes the operator of multiplication by $x_i$. Hence we may write
\begin{equation}\label{E:Hall345}
\begin{split}
J_r(\mathbf{1}^{\ss}_{1,d_1} \cdots \mathbf{1}^{\ss}_{1,d_r})
 &=\sum_{\sigma \in \mathfrak{S}_r}\prod_{(i,j) \in I_{\sigma}} h_{\E}\big( x_i/x_j\big) \sigma ( x_{1}^{d_1}  \cdots x_{r}^{d_r})\\
& =x_1^{d_1} \star \cdots \star x_1^{d_r}
 \end{split}
 \end{equation}
as wanted. This proves that $J$ is a morphism of algebras. The statement regarding the coproduct is a direct consequence of the
easily checked relation $$\omega^{\otimes r}\circ \widetilde{\Delta}_{(1^r)}= (\omega^{\otimes s} \otimes \omega^{\otimes t})\circ\widetilde{\Delta}_{(1^s),(1^t)} \circ (\omega \otimes \omega)\circ \widetilde{\Delta}_{s,t}.$$ \end{proof}

\vspace{.1in}

\noindent By Proposition~\ref{P:shuffle1} above, $\UU^>_{\E}$ is isomorphic to the subalgebra $\mathbf{S}_{h_{\E}(z)}$ of $\mathbf{F}_{h_{\E}(z)}$ generated by the degree one component $\mathbf{F}_{h_{\E}(z)}[1]=\CC[x_1^{\pm 1}]$. Observe that $\mathbf{S}_{h_{\E}(z)} \subset \mathbf{F}^{rat}_{h_{\E}(z)}$.

\vspace{.2in}

\addtocounter{theo}{1}
\paragraph{\textbf{Remark~\thetheo}} The triple $(\mathbf{F}_{h(z)}, \star, \Delta)$ is not a bialgebra, just as it is wrong that $( \UU^>_{\E}, \cdot, (\omega \otimes \omega) \circ \Delta)$ is a bialgebra. 
Rather $(\mathbf{F}_{h(z)}, \star, \Delta)$ 
becomes a bialgebra after a suitable twist. This is very similar to the twist introduced by Lusztig in \cite{Lusbook}.
Define a new multiplication $\star$ on $\mathbf{F}_{h(z)} \otimes
\mathbf{F}_{h(z)}$ by introducing a correcting factor
\begin{equation}\label{E:twistcoporod}
\big( a \otimes P(x_1, \ldots, x_r) \big) \star \big(Q(x_1, \ldots,
x_s) \otimes b\big)= \big(a \star Q'(x_1, \ldots, x_s)\bigr) \otimes \bigl(P'(x_1,
\ldots, x_r) \star b\bigr),
\end{equation}
where
$$Q'(x_1, \ldots, x_s) \otimes P'(x_1, \ldots, x_r)=\bigg(\prod_{\substack{1 
\leqslant i \leqslant r \\
1 \leqslant j \leqslant s}} h_X(x_j \otimes x_i^{-1}) \bigg)
Q(x_1, \ldots, x_s) \otimes P(x_1, \ldots, x_r).$$
It follows from (\ref{E:shuffle1}) and (\ref{E:coprodshuffle}) that
$(\mathbf{F}_{h(z)}, \star, \Delta)$ is a twisted bialgebra, i.e.,
we have  $\Delta(u \star v)=\Delta(u) \star \Delta(v)$ for any $u,v \in \mathbf{F}_{h(z)}$.
Moreover $(\mathbf{S}_{h(z)}, \star, \Delta)$ is a sub-bialgebra.

\vspace{.2in}

\addtocounter{theo}{1}
\paragraph{\textbf{Remark~\thetheo}} Formula (\ref{E:Hall345}) is essentially the Gindikin-Karpelevich identity, see e.g., \cite{GindiKarp}.

\vspace{.2in}

\paragraph{\textbf{1.10.}} In this section we rephrase the results of  Section~1.9 in more
convenient terms. Namely, we identify $\UU^>_{\E}$ with another shuffle algebra, but this time in a space of
\textit{symmetric} polynomials. This is precisely the setting of \cite{FO}, which we now explain.
Let $g(z) \in \CC(z)$ be a rational function. For $r
\geqslant 1$ we put $g(x_1, \ldots, x_r)=\prod_{i<j} g(x_i/x_j)$.
Let us denote by
\begin{equation*}
\begin{split}
Sym_r: \CC(x_1, \ldots, x_r) \to \CC(x_1, \ldots,
x_r)^{\mathfrak{S}_r},\quad P(x_1, \ldots, x_r) \mapsto \sum_{w \in
\mathfrak{S}_r} w P(x_1, \ldots, x_{r})
\end{split}
\end{equation*}
the standard symmetrization operator and let us consider
the weighted symmetrization
\begin{equation*}
\begin{split}
\Psi_r: \CC[x_1^{\pm 1}, \ldots, x_r^{\pm r}] \to \CC(x_1, \ldots,
x_r)^{\mathfrak{S}_r},\quad P(x_1, \ldots, x_r) \mapsto Sym_r \big(
g(x_1, \ldots, x_r) P(x_1, \ldots, x_r)\big).
\end{split}
\end{equation*}
Let $\mathbf{A}_r$ be the image of $\Psi_r$.
There is a unique map $m_{r,s}: \mathbf{A}_r \otimes \mathbf{A}_{s}
\to \mathbf{A}_{r+s}$ which makes the following diagram commute
\begin{equation}\label{E:shufflediagram}
\xymatrix{ \CC[x_1^{\pm 1}, \ldots, x_r^{\pm 1}] \otimes \CC[x_1^{\pm 1}, \ldots, x_{s}^{\pm 1}]
\ar[r]^-{\Psi_r \otimes \Psi_{s}} \ar[d]_-{i_{r,s}} & \mathbf{A}_r \otimes \mathbf{A}_{s} \ar[d]_-{m_{r,s}} \\
 \CC[x_1^{\pm 1}, \ldots, x_{r+s}^{\pm 1}] \ar[r]^-{\Psi_{r+s}} &
 \mathbf{A}_{r+s}}.
\end{equation}
Here $i_{r,s}$ is the obvious isomorphism 
$$i_{r,s}(P(x_1, \ldots, x_r) \otimes Q(x_1, \ldots, x_s))=P(x_1, \ldots, x_r)Q(x_{r+1}, \ldots, x_{r+s}).$$
The maps $m_{r,s}$
endow the space $$\Ab_{g(z)}=\CC 1 \oplus \bigoplus_{r \geqslant
1} \mathbf{A}_r$$ with the structure of an associative algebra whose
product is given by
\begin{equation}\label{E:shuffle2}
P(x_1, \ldots, x_r) \star Q(x_1, \ldots, x_{s})=\sum_{w \in
Sh_{r,s}} w  \bigg( \hspace{-.1in}\prod_{\substack{1 \leqslant i
\leqslant r\\ r+1 \leqslant j \leqslant r+s}}
\hspace{-.15in}g(x_i/x_j) P(x_1, \ldots, x_r) Q(x_{r+1},
\ldots, x_{r+s})\bigg).
\end{equation}
Note that, by construction, the algebra $\Ab_{g(z)}$ is generated by the
subspace $\mathbf{A}_1=\CC[x_1^{\pm 1}]$. As before, the
shuffle product $m_{r,s}$ may be extended to the whole space
$\bigoplus_r \CC(x_1, \ldots, x_r)^{\mathfrak{S}_r}$.

\vspace{.1in}

\begin{prop}\label{P:kloop} Let $g(z)$ be a rational function such that
$h(z)=g(z^{-1})/g(z)$.
Then the assignment $x_1^l \mapsto x_1^l$ in degree one extends to an algebra
isomorphism $\mathbf{S}_{h(z)}\to\Ab_{g(z)}$.
\end{prop}
\begin{proof} Consider the transformation
\begin{equation*}
\Xi~:\quad \Ab_{g(z)}
\to\mathbf{F}_{h(z)},\quad P(x_1, \ldots, x_r) \mapsto
g^{-1}(x_1,\dots,x_r) P(x_1, \ldots, x_r)
\end{equation*}
Comparing the definitions of the shuffle products (\ref{E:shuffle1})
and (\ref{E:shuffle2}) one checks that $\Xi$ is an algebra
embedding. It maps $\Ab_{g(z)}$ to $\mathbf{S}_{h(z)}$ since both are
generated in degree one. The proposition follows.
\end{proof}

\begin{cor}\label{C:kloop} Let $g_{\E}(z)$ be a rational function such that
\begin{equation}\label{E:gx(z)}
h_{\E}(z)=g_\E(z^{-1})/g_{\E}(z).
\end{equation} Then the map $\mathbf{1}^{\ss}_{1,l} \mapsto x^l_1$ 
extends to an algebra isomorphism
$\Upsilon_{\E}:\UU^>_{\E}\to\Ab_{g_X(z)}$.
\end{cor}

\vspace{.2in}

\noindent
Using the functional equation for zeta functions 
$$\zeta_{\E}(q^{-1}z)=z^{2(g-1)}q^{1-g}\zeta_{\E}(z^{-1})$$ one checks that $z^{g-1}\zeta_\E(z^{-1})$ is a solution of (\ref{E:gx(z)}). The same is also true of $z^{g-1}\zeta_{\E}(z^{-1})k(z)$ for any function $k(z)$ satisfying $k(z)=k(z^{-1})$. It will actually be more convenient for
us to set 
\begin{equation}\label{E:deftzetax(z)}
\tilde{\zeta}_{\E}(z)=\zeta_{\E}(z)(1-qz)(1-qz^{-1})=\frac{1-qz^{-1}}{1-z}\prod_{i=1}^g (1-\a_iz)(1-\bar\a_iz),
\end{equation}
\begin{equation}\label{E:defgx(z)}
g_{\E}(z)=z^{g-1}\tilde{\zeta}_{\E}(z^{-1}).
\end{equation}
From now on we fix the above choice for $g_{\E}(z)$ and simply write 
$\mathbf{A}=\Ab_{g_{\E}(z)}$. Thus we have an algebra isomorphism
$$\Upsilon_{\E}:\UU^>_{\E}\to\Ab.$$

\vspace{.2in}

\addtocounter{theo}{1}
\paragraph{\textbf{Remark~\thetheo}} 
The coproduct structure on $\UU^>_{\E}$ may also be written
down explicitly in terms of $\Ab$, but it is rather less pleasant than (\ref{E:coprodshuffle}). We won't need
this.

\vspace{.2in}

\paragraph{\textbf{1.11.}} So far, we have only considered the spherical Hall algebra $\UU_{\E}$
and its vector bundle part $\UU^>_{\E}$ for a \textit{fixed}
curve $\E$ of genus $g$. However, 
these only depend  on the zeta function of
$\E$. Equivalently, we may view $\UU_{\E}$ and
$\UU^>_{\E}$ as the specializations at $(\a_1, \bar\a_1,
\ldots, \a_g, \bar\a_g)$ of a ``universal'' algebra which
depends on the genus $g$ and on a point in the torus
\begin{equation}\label{E:torusa}
{T}_a=\{(\eta_1, \bar\eta_1, \ldots, \eta_g,\bar\eta_g) \in (\CC^\times)^{2g}\;;\; \eta_i \bar\eta_i=
\eta_j\bar\eta_j,\;\forall i,j\}\simeq (\CC^\times)^{g+1}.
\end{equation}
Let $$R_a=R_{T_a},\quad {K}_a=K_{T_a}$$ be the complexified representation ring of $T_a$
and its fraction field. We define, using
the shuffle presentation of Section~1.6,  some ${K}_a$-algebras
$\UU^>_{K_a}=\Ab_{K_a}$, $\UU_{{K}_a}$, $\dots$
The bialgebra structure and Green's bilinear form both depend
polynomially on $(\a_1, \bar\a_1, \ldots, \a_g,
\bar\a_g)$ and may hence be defined over ${K}_a$.
Let $\Ab_{R_a}$ be the $R_a$-subalgebra of $\Ab_{K_a}$ generated by 
$R_a[x_1^{\pm 1}] \subset \Ab_{K_a}[1]$.
By construction $\Ab_{R_a}$ is a torsion-free integral form of $\Ab_{K_a}$, in the sense that $\Ab_{R_a} \otimes_{R_a} K_a = \Ab_{K_a}$. Moreover there exists a natural specialization map 
$$\Ab_{R_a} \to  \UU^>_{\E}, \quad 
x_1^d \mapsto \mathbf{1}^{\ss}_{1,d},\quad d\in\ZZ,$$ for a fixed curve $\E$
of genus $g$. We'll write $\UU^>_{R_a}=\Ab_{R_a}$ to emphasize this link with Hall algebras.

\vspace{.2in}

\paragraph{\textbf{1.12.}} In this paragraph we partially describe the 
image of the constant term map in rank $r$, i.e.,
we try to determine $\mathbf{A}_r$ inside the space of all symmetric rational functions. For this we consider the action of
${\UU}^0_{\E}$ on $\UU^>_{\E}$ by means of the {Hecke operators}. 
Indeed, the presentation of $\UU^>_{\E}$ as
a shuffle algebra is particularly well suited to understand this
action. Define algebra homomorphisms
\begin{equation}\label{pi} 
\pi_r~: \UU^0_{\E} \to
\CC[x_1^{\pm 1}, \ldots, x_r^{\pm 1}]^{\mathfrak{S}_r},\quad T_{0,d}
\mapsto c_d p_d(x_1, \ldots, x_r)=c_d \big(\sum_i x_i^d\big), \quad
d\geqslant 0,
\end{equation} 
where $c_d=v^d\#\E(\mathbb{F}_{q^d})[d]/{d}$.
From Lemma~\ref{L:Hall2} one deduces the following result.

\vspace{.1in}

\begin{lem} For any $r \geqslant 1$ the map $\Upsilon_{\E}$ intertwines the Hecke action of $\mathbf{U}^0_{\E}$ on $\UU^>_{\E}[r]$ with the natural action of $\CC[x_1^{\pm 1}, \ldots, x_r^{\pm 1}]^{\mathfrak{S}_r}$ on $\mathbf{A}_r$, i.e., we have
$$\Upsilon_{\E}(u_0 \bullet v)=\pi_r(u_0)\cdot \Upsilon_\E(v),
\quad u_0 \in \UU^0_{\E},\quad v \in \UU^>_{\E}[r].$$
\end{lem}

\vspace{.1in}

\noindent In other words, the $\UU^0_{\E}$-module structure of $\UU^>_{\E}[r]$ is reflected in the structure of
$\Upsilon_{\E}(\UU^>_{\E}[r])=\mathbf{A}_r$ as a
$\CC[x_1^{\pm 1}, \ldots, x_r^{\pm r}]^{\mathfrak{S}_r}$-module. In particular,

\vspace{.1in}

\begin{cor}\label{C:torsionfree} 
The Hecke action of $\UU^0_{\E}$ on $\UU^>_{\E}$ is torsion free.
\end{cor}

\vspace{.1in}

\noindent The above definitions can be made for the Hall algebras $\UU^>_{K_a}$ over the field
$K_a$ as well, see Section 1.11.
The following proposition describes the structure of $\UU^>_{K_a}[r]$ as a 
$K_a[x_1^{\pm 1}, \ldots,x_r^{\pm 1}]^{\mathfrak{S}_r}$-module.
Set 
$$\Delta({x})=\prod_{i<j}(x_i-x_j).$$ 
It is easy to see that $\mathbf{A}_{K_a,r}$ is an ideal of $K_a[x_1^{\pm 1}, \ldots, x_r^{\pm 1}]^{\mathfrak{S}_r}$.

\begin{prop}\label{T:wheels} For $r \geqslant 1$ we have
$$\Delta({x})^{r!/2} I_r \subseteq \mathbf{A}_{K_a,r} \subseteq I_r.$$
Here $I_r \subset K_a[x_1^{\pm 1}, \ldots, x_r^{\pm 1}]^{\mathfrak{S}_r}$ is
the ideal sheaf of the reduced closed subvariety $Z_r \subset S^r(K_a^*)$ defined as
\begin{equation}\label{E:wheels}
Z_r=\bigcup_{\a} Z_r^{(\a)},\quad
Z_r^{(\a)}=\{ \mathfrak{S}_r (x_1, \ldots, x_r)\;;\; x_{1}=\a x_{2},\, x_{2}=\bar\a x_{3}\},
\end{equation}
where $\a$ runs into $\{\a_1, \bar\a_1, \ldots, \a_g, \bar\a_g\}$.
\end{prop}

\vspace{.1in}





\begin{proof} Recall that $\mathbf{A}_{K_a,r}$ is
the image of the operator $\Psi_r$ defined by
\begin{equation*}
\begin{split}
\Psi_r(P(x_1, \ldots, x_r))=&\\
=\sum_{\sigma \in \mathfrak{S}_r} \sigma  \bigg( \prod_{i<j} &
\bigg( \frac{x_i}{x_j}\bigg)^{g-1} \frac{(1-q^{-1}x_j/x_i)\prod_{l=1}^g (1-\a_l x_j/x_i)(1-\bar\a_lx_j/x_i)}{(1-x_j/x_i)}
P(x_1, \ldots, x_r) \bigg). 
\end{split}
\end{equation*}
We have
\begin{equation*}\label{E:whe1}
\Psi_r(P(x_1, \ldots, x_r))=\frac{(x_1 \ldots x_r)^{2(1-g)}(-q^{r(r-1)/2})}{\Delta({x})}
\sum_{\sigma} (-1)^{l(\sigma)} \sigma \big( \prod_{\substack{i<j \\ \gamma \in 
\Gamma}}  (x_i-\gamma x_j)
P(x_1, \ldots, x_r) \big),
\end{equation*}
where $\Gamma=\{ \a_1, \bar\a_1, \ldots, \a_g, \bar\a_g, q^{-1}\}$. Thus
$\mathbf{A}_{K_a,r}=Im\;\Psi'_r$ with $\Psi'_r$ defined by
$$\Psi'_r(P(x_1, \ldots, x_r))=\frac{1}{\Delta({x})}\sum_{\sigma} (-1)^{l(\sigma)} \sigma
\big( \prod_{\substack{i<j \\ \gamma \in \Gamma}}  (x_i-\gamma x_j) P(x_1, \ldots, x_r) \big).$$
Now let $\a \in \{\a_1, \bar\a_1, \ldots, \a_g, \bar\a_g\}$. For any $\sigma \in \mathfrak{S}_r$ we have either
$\sigma^{-1}(1)<\sigma^{-1}(2),$ or $\sigma^{-1}(2) < \sigma^{-1}(3),$ or $\sigma^{-1}(3) < \sigma^{-1}(1)$. 
This means that
for any $\sigma$ the element
$$X_{\sigma}({x})=\sigma \bigl(\prod_{i<j, \gamma \in \Gamma}(x_i-\gamma x_j)\bigr)$$ is divisible
by $x_{1}-\a x_{2}$, by $x_{2}-\bar\a x_{3}$ or by $x_{3}-q^{-1}x_{1}$. Of course, the same holds if
we replace $(1,2,3)$ by any other triplet of elements in $\{1, \ldots, r\}$. In all cases, $X_{\sigma}({x})$
belongs to the ideal sheaf of $Z_r^{(\a)}$. It follows that $$\Psi'_r (P(x_1, \ldots, x_r)) =
\Delta({x})^{-1} \sum_{\sigma} (-1)^{l(\sigma)}X_\sigma({x}) \sigma
P(x_1, \ldots, x_r)$$ belongs to that ideal sheaf as well, for any polynomial $P(x_1, \ldots, x_r)$. 
We have proved that $supp\; \mathbf{A}_{K_a,r} \supset Z_r$, i.e., that $\mathbf{A}_{K_a,r} \subset I_r$.

We now prove the reverse inclusion. Let
$\pi: (K_a^*)^r \to S^r K_a$ be the standard projection. It will be convenient to lift everything
to $(K_a^*)^r$ via $\pi$. Put $\tilde{\mathbf{A}}_{K_a,r}=\mathbf{A}_{K_a,r} \otimes K_a[x_1^{\pm 1}, \ldots, x_r^{\pm 1}]$,
where the tensor product is taken over ${K_a[x_1^{\pm 1}, \ldots, x_r^{\pm 1}]^{\mathfrak{S}_r}}$.
This is an ideal of $K_a[x_1^{\pm 1}, \ldots, \x_r^{\pm 1}]$, and we have
$$\mathbf{A}_{K_a,r}=\tilde{\mathbf{A}}_{K_a,r} \cap K_a[x_1^{\pm 1}, \ldots, x_r^{\pm 1}]^{\mathfrak{S}_r}.$$
The space $K_a[x_1^{\pm 1}, \ldots, x_r^{\pm 1}]$, viewed as a
$K_a[x_1^{\pm 1}, \ldots, x_r^{\pm 1}]^{\mathfrak{S}_r}$-module, is freely generated by the $r!$ monomials
$\{x_1^{n_1} \cdots x_r^{n_r}\;;\; 0 \leqslant n_i \leqslant r-i\},$ see e.g., \cite[Lemma~7.6.1]{Lascoux}. We deduce that
$\mathbf{A}_{K_a,r}$ is generated over $K_a[x_1^{\pm 1}, \ldots, x_r^{\pm 1}]^{\mathfrak{S}_r}$ by the $r!$
elements in
$$\{\Psi'(x_1^{n_1} \cdots x_r^{n_r})\;;\; 0 \leqslant n_i \leqslant r-i\},$$ 
and thus that $\tilde{\mathbf{A}}_{K_a,r}$ is
the ideal of $K_a[x_1^{\pm 1}, \ldots, x_r^{\pm 1}]$ generated by the set
$$\{\Psi'(x_1^{n_1} \cdots x_r^{n_r})\;;\; 0 \leqslant n_i \leqslant r-i\}.$$
Consider the $r! \times r!$ square matrix  $B=(b_{\sigma, I})_{\sigma,I}$ where $\sigma \in \mathfrak{S}_r$, 
$I\in \{(n_1, \ldots, n_r)\;;\; 0 \leqslant n_i \leqslant r-i\},$ with entries
$$b_{\sigma,I}=\frac{1}{\Delta({x})} (-1)^{l(\sigma)}\sigma(x_1^{n_1} \cdots x_r^{n_r}).$$
Let $\Psi$ be the column vector with entries $\Psi'_r(x_1^{n_1} \cdots x_r^{n_r})$ for $(n_1, \ldots, n_r) \in I$
and likewise let $X$ be the column vector with entries $X_{\sigma}({x})$ for $\sigma \in \mathfrak{S}_r$,
so that we have $\Psi=B \cdot X$.

\vspace{.1in}

\begin{lem} We have $det(B)= \pm \Delta({x})^{-r!/2}$.
\end{lem}
\begin{proof} Left to the reader. The sign depends on the particular choice of ordering for the rows or columns of $B$.
\end{proof}

\vspace{.1in}

\noindent
Thus, the matrix $B$ has an inverse $B^{-1}$ whose entries belong to  $\Delta({x})^{1-r!/2}K_a[x_1^{\pm 1}, \ldots, x_r^{\pm 1}]$.
As a consequence, $X_{\sigma}({x})$ belongs to 
$$\Delta({x})^{1-r!/2}\sum_{n_1, \ldots, n_r} K_a[x_1^{\pm 1}, \ldots,
x_r^{\pm 1}] \Psi'_r(x_1^{n_1} \cdots x_r^{n_r})=\Delta({x})^{1-r!/2}\tilde{\mathbf{A}}_{K_a,r}.$$
Let $\tilde{J}_r$ be the ideal of $ K_a[x_1^{\pm 1}, \ldots, x_r^{\pm 1}]$
generated by the $\{X_{\sigma}({x})\}_{\sigma}$. Let also
$$\tilde{I}_r=I_r \otimes K_a[x_1^{\pm 1}, \ldots, x_r^{\pm 1}]$$ be the (radical) ideal sheaf of 
$\pi^{-1}(Z_r)$. We will show that $\tilde{J}_r=\tilde{I}_r$ by checking that $supp\;\tilde{J}_r=\pi^{-1}(Z_r)$ and that $\tilde{J}_r$ is radical. 
Indeed, we have $$supp\;\tilde{J}_r =
\bigcap_{\sigma} supp\; X_{\sigma}({x}).$$
This last intersection is precisely given by the equations (\ref{E:wheels}) as we now show. Let
$(z_1, \ldots, z_r) \in (K_a^*)^r$ satisfy $X_{\sigma}(z_1, \ldots, z_r)=0$ for all $\sigma$. This means that for
every $\sigma$ there exists a pair $(i,j)$ and $\gamma \in \Gamma$ such that $\sigma(i)<\sigma(j)$ and
$z_i=\gamma z_j$. Let us draw an arrow $i \to j$ if $z_i/z_j \in \Gamma$. If the graph on $\{1, \ldots, r\}$
contains no oriented cycles then we can find a permutation $\sigma$ satisfying $i \to j \Rightarrow \sigma(i) > \sigma(j)$,
and this $\sigma$ violates the above condition. Hence there exists a cycle $i_1 \to i_2 \to \cdots \to i_l \to i_1$. Note
that we have $\frac{z_{i_1}}{z_{i_2}} \cdots \frac{z_{i_l}}{z_{i_1}}=1$. A sequence of elements of $\Gamma$ which
multiplies to one necessarily contains a subsequence $(\a,q^{-1})$. But this means that there are three indices $i,j,k$
such that $z_i=\bar\a z_j, z_j =\alpha z_k$ (recall that $\a \bar\a=q$). In other words, we have
$$(z_1, \ldots, z_r) \in
\pi^{-1}(Z_r),\quad\bigcap_{\sigma} supp\;X_{\sigma}({x})=\pi^{-1}(Z_r),$$ as wanted. 
We now check that $\tilde{J}_r$ is a radical ideal. Set $S_r=K_a[x_1^{\pm 1}, \ldots, x_r^{\pm 1}]/\tilde{J}_r$. We have to prove that no element of $S_r$ is
nilpotent. For this, it is enough to check that for each irreducible component $C$ of $\pi^{-1}(Z_r)$ there exists a point
$z \in C$ for which the localization $S_{r,z}$ of $S_r$ at $z$ has no nilpotent. We have
$$\pi^{-1}(Z_r)=\bigcup_{(\underline{i}, \a)} \tilde{Z}_r^{(\underline{i},\a)}, \qquad \tilde{Z}_r^{(\underline{i},\a)}=
\{(x_1, \ldots, x_r); x_{i_1}=\a x_{i_2}, x_{i_2}=\bar\a x_{i_3}\},$$
where $\underline{i}=(i_1,i_2,i_3)$ runs among all triplets of elements of $\{1, \ldots, r\}$ and 
$\a \in \{\a_1, \bar\a_1,
\ldots, \a_g \bar\a_g\}$. For a given $(\underline{i}, \a)$ 
as above let us pick a generic point $z=(z_1, \ldots, z_r)$ of
$\tilde{Z}_r^{(\underline{i},\a)}$. Hence $z$ satisfies $$z_{i_1}=\a z_{i_2},\quad z_{i_2}=
\bar\a z_{i_3},\quad z_{i_3}=q^{-1}z_{i_1}$$
but $z_i/z_j \not\in \Gamma$ for any other value of $(i,j)$. 
The functions $(x_i-\gamma x_j)$ are therefore invertible in $S_{r,z}$
except for the three values $(i,j,\gamma)$ above. 
We may choose a permutation $\sigma \in \mathfrak{S}_r$ for which
$\sigma^{-1}(i_3)< \sigma^{-1}(i_1) <\sigma^{-1}(i_2)$. Then we have
$$X_{\sigma}({x})=(x_{i_1}-\a x_{i_2}) Y({x})$$
for some polynomial $Y({x})$ which is invertible in $S_{r,z}$. But this means that $x_{i_1}-\a x_{i_2}=0$ in $S_{r,z}$.
A similar argument shows that $x_{i_2}-\bar\a x_{i_3}=x_{i_3}-q^{-1}x_{i_1}=0$ in $S_{r,z}$ as well. We deduce that
$$S_{r,z}=K_a[x_1^{\pm 1}, \ldots, x_r^{\pm 1}]_{z}/\langle x_{i_1}-\a x_{i_2}, x_{i_2}-\bar\a x_{i_3}\rangle.$$ In particular
$S_{r,z}$ has no nilpotent element. We have proved that $\tilde{J}_r$ is radical. 
We have shown that $$\tilde{\mathbf{A}}_{K_a,r} \supset \Delta({x})^{r!/2-1} \tilde{I}_r.$$ Taking the intersection with
$K_a[x_1^{\pm 1}, \ldots, x_r^{\pm 1}]^{\mathfrak{S}_r}$ yields the desired inclusion.
This finishes the proof of Proposition~\ref{T:wheels}.
\end{proof}

\vspace{.2in}

\addtocounter{theo}{1}

\noindent
\textbf{Remark \thetheo.} 
$(a)$ Equations (\ref{E:wheels}) are dubbed \textit{wheel conditions} in \cite{Feigin}, in the general context
of Feigin-Odesskii algebras. A weaker form of Proposition~\ref{T:wheels} 
appears in \cite{FT} for $g=1$. Note that $Z_r$ is the image under the 
natural map $(K_a^*)^r \to S^r K_a^*$ of a union of
$2g$ codimension two subspaces. Observe also that $Z_r$ is empty for $r=2$ but not for any $r>2$. As a consequence
$\mathbf{A}_{K_a,r}$ is \textit{not} free over $K_a[x_1^{\pm 1}, \ldots, x_r^{\pm 1}]^{\mathfrak{S}_r}$ as soon as $r>2$ and $g >0$.

\vspace{1mm}

$(b)$ The above proposition holds for the Hall algebras $\UU^>_{K_a}$ only, 
or alternatively when the curve $\E$ is \textit{generic}.
For special curves the Frobenius eigenvalues might not all be distinct (for instance, some of them could be
 real); this corresponds
to the merging of two irreducible components $Z_r^{(\a)}, Z_r^{(\a')}$.

\vspace{.2in}

\paragraph{\textbf{1.13.}} In this last section, we define a twisted version $\dot{\UU}^>_{\E}$ of 
the spherical Hall algebra $\UU^>_{\E}$ of $\E$. This twisted version is more relevant
to the geometric Langlands program.
The operation of tensoring with a line bundle is an exact autoequivalence
of the category $Coh(\E)$ and it induces an automorphism of the Hall algebra
$\H_{\E}$. Since 
$$\mathcal{L} \otimes \mathbf{1}^{\ss}_{1,d}=
\mathbf{1}^{\ss}_{d+deg(\mathcal{L})},\quad d\in\ZZ,$$
tensoring with a line bundle preserves the subalgebra $\UU^>_{\E}$. 
Let $\Omega \in Pic^{2(g-1)}(\E)$ be the canonical bundle of $\E$. 
We define a twisted multiplication in $\UU^>_{\E}$ by the rule
\begin{equation}\label{E:twistmult}
u_r \circ w_s=(u_r \otimes \Omega^{-s/2}) \cdot (w_s \otimes \Omega^{r/2}),
\quad u_r\in \UU^>_{\E}[r], \quad w_s \in \UU^>_{\E}[s]. 
\end{equation}
The new multiplication $\circ$ is associative. We denote by $\dot{\UU}^>_{\E}$ the ensuing
algebra. 
Let us rewrite the new multiplication in terms of the shuffle algebra $\Ab=\mathbf{A}_{g_{\E}(z)}$. We have
$$J_r(u_r \otimes \Omega^{n/2})=(x_1 \cdots x_r)^{(g-1)n} J_r(u_r),\quad
n \in \Z,\quad u_r \in \UU^>_{\E}[r].$$
Hence the twisted multiplication in $\mathbf{A}_{g_{\E}(z)}$ is given by
\begin{equation}
\begin{split}
&P(x_1, \ldots, x_r) \circ Q(x_1, \ldots, x_s)=\big( (x_1 \cdots x_r)^{(1-g)s} P(x_1, \ldots, x_r)\big) \star
\big( (x_1 \cdots x_s)^{(g-1)r} Q(x_1, \ldots, x_s)\big)\\
&=\sum_{\sigma \in Sh_{r,s}} \sigma\bigg\{ \prod_{\substack{1 \leqslant i 
\leqslant r\\ r+1 \leqslant j \leqslant s}} 
\hspace{-.1in} g_{\E}(x_i/x_j) (x_1 \cdots x_r)^{(1-g)s}(x_{r+1} \cdots x_{r+s})^{(g-1)r} 
P(x_1, \ldots, x_r) Q(x_{r+1}, \ldots, x_{r+s}) \bigg\}\\
&=\sum_{\sigma \in Sh_{r,s}}\sigma \bigg\{ \prod_{\substack{1 \leqslant i 
\leqslant r\\ r+1 \leqslant j \leqslant s}} \hspace{-.1in} 
g'_{\E}(x_i/x_j) P(x_1, \ldots, x_r) Q(x_{r+1}, \ldots, x_{r+s}) \bigg\}
\end{split}
\end{equation}
with $g'_{\E}(z)=z^{(1-g)}g(z)=\tilde{\zeta}_{\E}(z^{-1})$. In other terms, we have the following.

\vspace{.1in}

\begin{cor}\label{C:finalshuffle} 
The assignement $\mathbf{1}^{\ss}_{1,d} \mapsto x_1^d$, $d\in\ZZ$, extends to
an algebra isomorphism 
$$\dot{\Upsilon}_{\E}: \dot{\UU}^>_{\E} {\to} \mathbf{A}_{\tilde{\zeta}_{\E}(z^{-1})}$$
such that
$\dot\Upsilon_{\E}(u_0 \bullet v)=\pi_r(u_0)\cdot\dot \Upsilon_\E(v)$ for
$u_0 \in \UU^0_{\E}$ and $v \in\dot \UU^>_{\E}[r].$
\end{cor}

\vspace{.1in}

\noindent We also have identifications of integral and rational forms 
$$\dot{\UU}^{>}_{R_a} \simeq \mathbf{A}_{R_a,\tilde{\zeta}(z^{-1})},
\quad\dot{\UU}^>_{K_a} \simeq\mathbf{A}_{K_a,\tilde{\zeta}(z^{-1})}. $$
 
\vspace{.2in}
 
\addtocounter{theo}{1}

\paragraph{\textbf{Remark \thetheo}} 
$(a)$ The twist (\ref{E:twistmult}) is classical, see e.g., 
\cite{Gaitsgory}. It may be written in a slightly more canonical fashion 
as follows. Identify, for $r \geqslant 1$, the algebra
$\CC[x_1^{\pm 1}, \ldots, x_r^{\pm 1}]^{\mathfrak{S}_r}$ with $\CC[T]^{W}$, where $T \subset G=GL_r$ is a maximal torus and $W=\mathfrak{S}_r$ is the Weyl group of $(G,T)$. This way we view the multiplication in $\mathbf{A}_{g(z)}$ as a collection of (induction) maps 
$$m_{L,G}~: \CC[T_L]^{W_L} \to \CC[T_G]^{W_G},$$ where $L$ is the Levi factor of some parabolic subgroup $P \subset G$ and $G$ runs among all $GL_r$. Then the twisted multiplication can be written as 
$$\dot{m}_{L,G}(f)=m_{L,G}(e^{2(1-g)(\rho_G-\rho_L)}f),$$ where  $\rho_H$ is half the sum of positive roots  of a reductive group $H$.

\vspace{1mm}

$(b)$ The main reason for considering $\dot{\UU}^>_{\E}$ rather than $\UU^>_{\E}$ is that the functional
equation for Eisenstein series takes in $\dot{\UU}^>_{\E}$ 
the following particularly nice form. 
Indeed, the series
$$\dot{E}(z_1, \ldots, z_r)=E_1(z_1) \circ \cdots \cdot \circ E_1(z_r),\quad
\quad E_1(z)=\sum_{d \in \Z} \mathbf{1}_{1,d}z^d,$$ 
viewed as a rational function in $z_1, \ldots, z_r$, is \textit{symmetric}.

\vspace{.2in}

\section{K-theoretic Hall algebras}

\vspace{.2in}

\paragraph{\textbf{2.1.}}
Let $G$ be a complex linear algebraic group. By a variety we'll
always mean a quasi-projective complex variety. We call $G$-variety
a variety $X$ with a rational action of $G$. Let $K^G(X)$ be the
complexified Grothendieck group of the Abelian category of the
$G$-equivariant coherent sheaves over $X$.
Let $P\subset G$ a parabolic subgroup and $H\subset P$ a Levi
subgroup. Fix a $H$-variety $Y$.  The group $P$ acts on $Y$ through
the obvious group homomorphism $P\to H$.
The induced $G$-variety is
$$X=G\times_PY.$$
Now assume that $Y$ is smooth. Given a smooth subscheme $O\subset Y$
let $T^*_OY\subset T^*Y$ be the conormal bundle to $O$. It is
well-known that the induced $H$-action on $T^*Y$ is Hamiltonian and
that the zero set of the moment map is the closed $H$-subvariety
$$T^*_HY=\bigcup_{O}T^*_OY\subset T^*Y,$$
where $O$ runs over the set of $H$-orbits. The following lemma is
left to the reader.

\begin{lem}
We have $T^*X=T^*_P(G\times Y)/P$ and $T^*_GX=G\times_PT^*_HY$. The
induction yields a canonical isomorphism $K^H(T^*_HY)=K^G(T^*_GX)$.
\end{lem}

\noindent We'll call fibration a smooth morphism which is locally trivial in
the Zariski topology. Let $X'$ be a smooth $G$-variety and $V$ be a
smooth $H$-variety. Assume that we are given $H$-equivariant
homomorphisms
$$\xymatrix{Y&V\ar[l]_p\ar[r]^q&X'}$$
such that $p$ is a fibration and $q$ is a
closed embedding. Set
$$W=G\times_PV$$ and consider the following maps
$$\gathered
\xymatrix{X&W\ar[l]_f\ar[r]^g& X'}\cr
f: (g,v)\, \mod\, P\mapsto (g,p(v))\, \mod\, P,\quad
g: (g,v)\, \mod\, P\mapsto gq(v).
\endgathered$$ The following properties are immediate.

\begin{lem}
The map $f$ is a $G$-equivariant fibration, the map $g$ is a
$G$-equivariant proper morphism, and the map $(f,g)$ is a closed
embedding $W\subset X\times X'$. The varieties $V$, $W$, $X$, $X'$
are smooth.
\end{lem}

\noindent We'll identify $W$ with its image in $X\times X'$. Let
$Z=T^*_W(X\times X')$ be the conormal bundle. It is again a smooth
$G$-variety. The obvious projections yield $G$-equivariant maps
$$\xymatrix{T^*X&Z\ar[l]_-\phi\ar[r]^-\psi&T^*X'}.$$ Consider the $G$-variety
$$Z_G=Z\cap(T^*_GX\times T^*_GX').$$
Recall that a morphism of varieties $S\to T$ is called regular if it
is the composition of a regular immersion $S\subset S'$, i.e., an
immersion which is locally defined by a regular sequence, and of a
smooth map $S'\to T$. Note that a regular map has a finite
tor-dimension and that a morphism $S\to T$ is regular whenever $S$,
$T$ are smooth.
The following is immediate.

\begin{lem}
(a) The map $\psi$ is proper and regular, the map $\phi$ is regular.

(b) We have $\phi^{-1}(T^*_GX)=Z_G$
and $\psi(Z_G)\subset T^*_GX'$.
\end{lem}

\vskip3mm

\noindent We'll abbreviate
$\phi_G=\phi|_{Z_G}$ and $\psi_G=\psi|_{Z_G}$.  We have the following diagram
$$\xymatrix{T^*_GX&Z_G\ar[l]_-{\phi_G}\ar[r]^-{\psi_G}&T^*_GX'}.$$
Recall that for any $G$-variety $M$ and any closed
$G$-stable subvariety $N\subset M$ the direct image by the obvious
inclusion $N\to M$ identifies $K^ G(N)$ with the complexified
Grothendieck group $K^G(M\on N)$ of the category of $G$-equivariant
coherent sheaves over $M$ supported on $N$. Since the map $\psi$ is
a proper morphism the derived direct image yields maps
$$\gathered
R\psi_*:K^G(Z)\to K^G(T^*X'),\cr
R\psi_*:K^G(Z_G)=K^G(Z\on Z_G)\to K^G(T^*X'\on T^*_GX')=K^G(T^*_GX').
\endgathered$$
Since the map $\phi$ has a finite tor-dimension the derived
pull-back yields maps
$$\gathered
L\phi^*:K^G(T^*X)\to K^G(Z),\cr
L\phi^*:K^G(T^*_GX)=K^G(T^*X\on T^*_GX)\to K^G(Z\on Z_G)=K^G(Z_G).
\endgathered$$
By definition $L\phi^*$ is the composition of the pull-back by the
projection $T^*X\times T^*X'\to T^*X$ and the derived pull-back by
the regular immersion $Z\subset T^*X\times T^*X'$.
Composing $R\psi_*$ and $L\phi^*$ we get a map
\begin{equation}\label{CO:00}R\psi_*\circ L\phi^*:K^G(T^*_GX)\to K^G(T^*_GX').
\end{equation}
By Lemma 2.1 the induction yields also an isomorphism
$$K^H(T^*_HY)=K^G(T^*_GX).$$ Composing it by $R\psi_*\circ L\phi^*$
we obtain a map
\begin{equation}\label{CO:0}K^H(T^*_HY)\to K^G(T^*_GX').\end{equation}

\vskip3mm

\paragraph{\textbf{2.2.}}
We'll apply the general construction recalled above to the following
particular case. First, let us fix some notation. Let $E$ be a
finite dimensional $\CC$-vector space. We'll abbreviate
$$G_E=GL(E),\quad\gen_E=\End(E).$$ We set
$$\gathered
C_E=\{(a,b)\in(\gen_E)^g\times(\gen_E)^{g}; \sum_r
[a_r,b_r]=0\},\cr a=(a_1,a_2,\dots a_g),\quad b=(b_1,b_2,\dots
b_g).\endgathered$$ If no confusion is possible we'll write $C=C_E$,
$G=G_E$ and $\gen=\gen_E$. Let also
\begin{equation}\label{E:Toruss}
\gathered
T_s=\{(h,e,f)\in(\CC^\times)^{2g+1}; e_r f_r=h,\ \forall r\},\cr
e=(e_1,e_2,\dots e_g),\quad f=(f_1,f_2,\dots
f_g).\endgathered
\end{equation}
Thus $T_s$ is a $g+1$-dimensional
torus which acts on $C$ in the following way
\begin{equation}\label{E:taction1}
(h,e,f)\cdot(a,b)=(e_1a_1,e_2a_2,\dots,f_1b_1,\dots f_gb_g).
\end{equation}
We may abbreviate $(e\cdot a,f\cdot b)$ for the right hand side. We also equip
$C$ with the diagonal $G$-action such that $G$ acts on $\gen$ by the
adjoint action. The $T_s$-action and the $G$-action on $C$ commute,
yielding an action of the group $T_s\times G$. Let $R_s$ be the
complexified representation ring of $T_s$. We have
$$\begin{gathered}
R_s=\CC[p^{\pm 1},x_1^{\pm 1},y_1^{\pm 1},\dots,x_g^{\pm 1}, y_g^ {\pm
1}]/(p-x_r y_r),\cr p(h,e,f)=h,\quad x_r(h,e,f)=e_r,\quad
y_r(h,e,f)=f_r.\end{gathered}$$

\vspace{.2in}

\paragraph{\textbf{2.3.}}
Next we fix a flag
$$0\to E_1\to E\to E_2\to 0.$$
Set $H=G_{E_1}\times G_{E_2}$ and $P=\{g\in G;g(E_1)=E_1\}$.  Let
$\hen$ and $\pen$ be the corresponding Lie algebras. Put
$Y=\hen^{g}$, $V=\pen^{g}$, and $X'=\gen^{g}$. The $G$-action on
$X'$ and the $H$-action on $Y$ are the adjoint ones. Put also
$$C_\gen=C_E,\quad C_\hen=C_{E_1}\times C_{E_2},\quad
C_\pen=(\pen^g\times\pen^{g})\cap C_E.$$ Let $a\mapsto p(a)=a_\hen$
denote the canonical maps $\pen\to\hen$ and $\pen^g\to\hen^g$. We
apply the general construction in Section 2.2 to the diagram
$$\xymatrix{Y&V\ar[l]_p\ar[r]^q&X'}$$
where $q$ is the obvious inclusion. We have the following lemma.
Here the $P$-actions on $\pen^g\times\pen^g$ and on
$\pen\times\hen^g\times\hen^g$ are the obvious ones, the $G$-action
on $T^*X'$ is as in Section 2.2, and the $G$-action on $T^*X$, $Z$
is by left multiplication on $G$. Further we identify $\gen^*=\gen$
and $\hen^*=\hen$ via the trace.

\begin{lem}\label{lem:Comm}
(a) There are isomorphisms of $G$-varieties
$$\begin{gathered}
T^*X'=\gen^g\times\gen^{g},\quad
Z=G\times_P(\pen^g\times\pen^{g}),\cr
T^*X=G\times_P\{(d,a,b)\in\pen\times\hen^g\times\hen^g;\,
d_\hen=\sum_r[a_r,b_r]\}.\end{gathered}$$ For each
$(a,b)\in\pen^g\times\pen^g$ we have
$$\begin{gathered}
\phi((g,a,b)\ \mod\ P)=(g,\sum_r[a_r,b_r],a_\hen,b_\hen)\ \mod\
P,\cr \psi((g,a,b)\ \mod\ P)=(gag^{-1},gbg^{-1}).
\end{gathered}$$

(b) There are isomorphisms of $G$-varieties
$$\begin{aligned}
T^*_GX'=C_\gen,\quad Z_G=G\times_PC_\pen,\quad
T^*_GX=G\times_PC_\hen.
\end{aligned}$$
The maps $\phi_G$, $\psi_G$ are the obvious ones.
\end{lem}

\begin{proof}
The linear map
$\delta:\pen\to\hen\times\gen$, $a\mapsto(a_\hen,a)$
is $P$-equivariant.
We'll identify $\pen$ and $\delta\pen$ whenever needed.
By Lemma 2.1 we have
$$\begin{aligned} T^*(X\times X')
&=T^*_P(G\times\hen^g\times\gen^g)/P\\
&=G\times_P\{(a,f)\in(\hen^g\times\gen^g)\times(\gen\times\hen^g\times\gen^g)^*;
\,f(-b,[\delta b,a])=0,\forall b\in \pen\},\cr T^*X
&=T^*_P(G\times\hen^g)/P\\
&=G\times_P\{(a,f)\in\hen^g\times(\gen\times\hen^g)^*;\,
f(-b,[b_\hen,a])=0,\forall b\in \pen\},\\
T^*X'&=\gen^g\times(\gen^*)^{g}=\gen^g\times\gen^g.
\end{aligned}$$
Fix  $f=(\lambda,\a)\in\gen^*\times(\hen^g)^*$. For each
$a\in\hen^g$ we have
$$f(-b,[b_\hen,a])=0,\ \forall b\in\pen\iff\lambda(b)=\a([b_\hen,a]),\
\forall b\in\pen.$$ Let $\pen_{nil}\subset\pen$ be the nilpotent
radical. Thus the left hand side is satisfied iff we have
$$\lambda(\pen_{nil})=0,\quad
\lambda|_\hen=\sum_r\ad^*(a_r)(\a_r),$$ where $\ad^*$ is the
coadjoint action. Under the canonical isomorphism $\gen^*\to\gen$
this yields the formula for $T^*X$ in the lemma. Next we have
$W=G\times_P\pen^g$ and $X\times X'=G\times_P(\hen^g\times\gen^g)$.
Further the inclusion $W\subset X\times X'$ is given by the map
$\delta$. We have also
$$\begin{aligned}
T^*W&=T^*_P(G\times\pen^g)/P\\
&=G\times_P\{(a,f)\in\pen^g\times(\gen\times\pen^g)^*;\,
f(-b,[b,a])=0,\forall b\in \pen\}.\end{aligned}$$ Let
$\pen^\perp\subset(\hen\times\gen)^*$ be the orthogonal of
$\delta\pen$. We get the following isomorphisms
$$\begin{gathered}
T^*(X\times X')|_W=
G\times_P\{(a,f)\in\pen^g\times(\gen\times\hen^g\times\gen^g)^*;\,
f(-b,\delta[b,a])=0,\forall b\in \pen\},\\
Z=T^*_W(X\times X')
=G\times_P(\pen^g\times(\pen^\perp)^g).
\end{gathered}$$
The canonical isomorphism
$(\hen\times\gen)^*\to\hen\times\gen$ identifies $\pen^\perp$
with
$$\pen'=\{(-a_\hen,a)\in\hen\times\gen;a\in\pen\}.$$
Note that $\pen'=\pen$ as a $P$-module. This yields an isomorphism
$$Z=G\times_P(\pen^g\times\pen^g).$$
The inclusion $Z\subset T^*(X\times X')$ is given
by the map $\delta:\pen^g\to\hen^g\times\gen^g$ and by the map
$$\pen^g=(\pen')^g=\{0\}\times(\pen^\perp)^g\subset
\{0\}\times(\hen^g\times\gen^g)^*\subset(\gen\times\hen^g\times\gen^g)^*.$$
The maps $\phi$, $\psi$ are composed of the inclusion $Z\subset
T^*(X\times X')$ and the projections to $T^*X$, $T^*X'$. Fix
$a,b\in\pen^g$. Consider the element $\xi\in Z$ equal to $(g,a,b)$
modulo $P$. We may identify $a$ with $\delta a=(a_\hen,a)$, which is
an element of $\hen^g\times\gen^g$, and $b$ with the element
$(-b_\hen,b)\in(\pen')^g$, which can be regarded as an element in
$$(\pen^\perp)^g\subset(\hen^g\times\gen^g)^*=
\{0\}\times(\hen^g\times\gen^g)^*\subset(\gen\times\hen^g\times\gen^g)^*
.$$ So $\xi$ can be viewed as an element in $T^*(X\times X')$, and
we have
$$\psi(\xi)= (gag^{-1},gbg^{-1}),\quad
\phi(\xi)= (g,-[a_\hen,b],a_\hen,-b_\hen)\ \mod\ P.$$ This
yields the formulas for $\phi$ and $\psi$ in the lemma.
The claim $(b)$ is left to the reader.
\end{proof}

\vspace{.2in}

\noindent Next, we set $\mathbf C_{E}=K^{T_s\times G}(C_\gen)$. A vector space
isomorphism $E\simeq E'$ yields an $R_s$-module isomorphism $\mathbf
C_{E}\simeq \mathbf C_{E'}.$ Let
$\mathbf C=\ind_E {\mathbf C}_{E}$,
where the colimit runs over the groupoid formed by all finite
dimensional vector spaces with their isomorphisms. There is a
$T_s\times G$-action on $T^*X$, $T^*X'$ such that $G$ acts as above
and $T$ acts by
$$\begin{gathered}
(h,e,f)\cdot(g,d,a,b)\ \mod\ P=(g,h\cdot d,e\cdot a,f\cdot b)\ \mod\
P,\cr (h,e,f)\cdot(g,a,b)\ \mod\ P=(g,e\cdot a,f\cdot b)\ \mod\
P.\end{gathered}$$ Here the symbol $h\cdot d$ is simply the
multiplication of $d$ by the scalar $h$. We define as in
(\ref{CO:0}) a $R_s$-linear homomorphism
$$K^{T_s\times H}(C_\hen)\to K^{T_s\times G}(C_\gen).$$ 
By the Kunneth formula \cite[Chap.~5.6.]{CG}, it can be
viewed as a map
\begin{equation}\label{CO:3}\mathbf C_{E_1}\otimes_{R_s} 
\mathbf C_{E_2}\to \mathbf C_{E}.\end{equation}
The following is proved as in \cite[Proposition 7.5]{SV2}.

\vspace{.1in}

\begin{prop}  The map (\ref{CO:3}) equips $\mathbf C$
with the structure of a $R_s$-algebra with 1.
\end{prop}

\vspace{.2in}

\paragraph{\textbf{2.4.}} We'll call the $R_s$-algebra $\Cb$ the {\it K-theoretic Hall algebra}.
It is naturally $\N$-graded, with the piece of degree $n$ equal to
the colimit over the groupoid formed by all $n$-dimensional vector
spaces with their isomorphisms
$\mathbf C_n=\ind_E \mathbf C_{E}$.
The \textit{spherical subalgebra} of $\Cb$ is the $R_s$-subalgebra
$\Vb\subset\Cb$ generated by $\Cb_1$. We'll abbreviate
$$\Vb_n=\Cb_n\cap\Vb,\quad\Vb_{K}=\Vb\otimes_{R_s}K_s,
\quad\Cb_{K}=\Cb\otimes_{R_s}K_s,$$
where $K_s$ is the fraction field of $R_s$. Write also $R_G$ for the complexified representation ring of
$T_s\times G$ and $K_G$ for its fraction field. For each $E$ as above
the direct image by the obvious inclusion
$C_\gen\subset\gen^g\times\gen^{g}$ yields a $R_{G}$-module
homomorphism
\begin{equation}\label{CO:14}\Cb_n\to K^{T_s\times
G}(\gen^g\times\gen^{g}).\end{equation} We conjecture that
(\ref{CO:14}) is an injective map. This conjecture is equivalent to
the following one. Set $n=\dim\,E$ and $G=GL_n$.

\begin{conj}\label{C:torsion} The
$R_{GL_n}$-module $\Cb_n$ is torsion-free.
\end{conj}

\noindent The kernel of (\ref{CO:14}) is the torsion $R_{GL_n}$-submodule
of $\Cb_n$. Let $\bar\Cb_n$ be the image of $\Vb_n$
by (\ref{CO:14}). We set
$$\bar\Cb=\bigoplus_{n\geqslant 0}\bar\Cb_n.$$

\begin{prop} The map (\ref{CO:14}) yields a surjective 
$R$-algebra homomorphism $\Vb \to \bar\Cb$.
\end{prop}

\begin{proof} For any finite dimensional vector space $E$
let $\tilde\Cb_E$ be the quotient of $\Cb_E$ by is torsion $R_{G_E}$-submodule.
Given  $E_1,E_2,E$ as in Section 2.3, we must check that the map 
(\ref{CO:3}) fits into a commutative square
$$\xymatrix{
\Cb_{E_1}\otimes_{R_s}\Cb_{E_2}\ar[r]\ar[d]&\Cb_E\ar[d]\cr
\tilde\Cb_{E_1}\otimes_{R_s}\tilde\Cb_{E_2}\ar[r]&\tilde\Cb_E.}
$$ 
The upper arrow is identified with the map
$K^{T_s\times H}(C_\hen)\to K^{T_s\times G}(C_\gen)$ 
which comes from (\ref{CO:00}). Further
$\tilde\Cb_{E_1}\otimes_{R_s}\tilde\Cb_{E_2}$ and $\tilde\Cb_E$ are identified with
the images by the obvious maps
$$K^{T_s\times G}(T^*_GX)\to K^{T_s\times G}(T^*X),\quad
K^{T_s\times G}(T^*_GX')\to K^{T_s\times G}(T^*X')$$
respectively.
So the proposition follows from the commutativity of the square
$$\xymatrix{
K^{T_s\times G}(T^*_GX)\ar[r]\ar[d]&K^{T_s\times G}(T^*_GX')\ar[d]\cr
K^{T_s\times G}(T^*X)\ar[r]&K^{T_s\times G}(T^*X').}
$$ 
\end{proof}

\vspace{.2in}

\paragraph{\textbf{2.5.}}
Fix a $n$-dimensional vector space $E$. Let $H\subset G$ be the
torus consisting of the diagonal matrices. The inverse image by the
obvious inclusion $i:\{0\}\to\gen^g\times\gen^{g}$ yields a map
$$Li^*:K^{T_s\times G}(\gen^g\times\gen^{g})\to R_{G}.$$
Composing it with (\ref{CO:14}) we get a $R_G$-linear map
$$\gamma_G:\Cb_n\to R_{G}.$$
In the same way we have a $R_H$-linear map
$$\gamma_H=(\gamma_{\CC^\times})^{\otimes n}:(\Cb_1)^{\otimes n}\to R_{H}$$
(the tensor power over $R_s$). Write
$$R_{\CC^\times}=R_s[z^{\pm 1}],\quad R_{
H}=(R_{\CC^\times})^{\otimes n}=R_s[z^{\pm 1}_1,z_2^{\pm
1},\dots,z_n^{\pm 1}],\quad R_{G}=R_s[z^{\pm 1}_1,z_2^{\pm
1},\dots,z_n^{\pm 1}]^{\mathfrak S_n}.$$
Note that we have 
$$K_{\CC^\times}=K_s(z),\quad K_{
H}=K_s(z_1,z_2,\dots,z_n),\quad K_{ G}=K_s(z_1,z_2,\dots,z_n)^{\mathfrak
S_n}.$$ Recall the standard symmetrization operator
$Sym_n: K_H\to K_G$.

\vspace{.1in}

\begin{prop}\label{P:KHall} We have the following commutative diagram
$$\xymatrix{
(\Cb_1)^{\otimes
n}\ar[r]^{\mu_n}\ar[d]_{\gamma_H}&\Cb_n\ar[d]^{\gamma_G}\cr
R_{H}\ar[r]^{\nu_n}&R_{G},}
$$ where the upper map is the multiplication in $\Cb$ and
the map $\nu_n$ is given by
\begin{equation}\label{E:KHallform}
\nu_n(P(z_1, \ldots, z_n))=Sym_n\bigl(k(z_1,z_2,\dots z_n)P(z_1,z_2,\dots
z_n)\bigr),
\end{equation}
where
\begin{equation}\label{E:Kg(z)}
\begin{gathered}
k(z_1, \ldots, z_n)=\prod_{i<j} k(z_i/z_j), \cr k(z)=
(1-z)^{-1}(1-p^{-1}z^{-1})\prod_r(1-x_r^{-1}z)(1-y_r^{-1}z).
\end{gathered}
\end{equation}
\end{prop}

\begin{proof} Fix a monomial $\theta_m=z_1^{m_1}z_2^{m_2}\cdots z_n^{m_n}$
with $m=(m_1,m_2,\dots m_n)\in\Z^n$. We'll regard $\theta_m$ as an
element of $R_H$. Since $C_\hen$ is a vector space and $T_s\times H$
is a torus, the $R_s$-module $K^{T_s\times H}(C_\hen)$ is spanned by the
classes of the $T_s\times H$-equivariant line bundles
$\Oc_{C_\hen}\langle m\rangle$. Here the symbol $\langle m\rangle$
means the tensor product of $\Oc_{C_\hen}$, with the trivial action,
by the character $\theta_m$. Note that
\begin{equation}\label{CO:24}\gamma_H(\Oc_{C_\hen}\langle
m\rangle)=\theta_m.\end{equation} Let $B\subset G$ be the Borel
subgroup consisting of upper triangular matrices. Let $\ben$ be the
Lie algebra of $B$ and $\nen$ be the nilpotent radical of $\ben$.
Recall that we have
$$T^*_GX=G\times_BC_\hen,\quad
T^*X=G\times_B(\nen\times C_\hen),\quad
C_\hen=\hen^g\times\hen^{g}.$$ Let $\Gamma$ denote the induction of
equivariant sheaves from $T_s\times B$ to $T_s\times G$. Consider the
$T_s\times G$-equivariant line bundle over $T^*_GX$
$$\Oc_{T^*_GX}(m)=\Gamma(\Oc_{C_\hen}\langle m\rangle).$$
For a future use, let us consider the following
commutative diagram
$$\xymatrix{
T^*_GX\ar[d]_j&&T^*_GX'\ar[d]_{j'}\cr
T^*X&Z\ar[r]^\psi\ar[l]_-\phi&T^*X'\cr
&G/B\ar[u]_i\ar[r]^\pi&\{0\}.\ar[u]_{i'}}$$
Here $j$, $j'$, $i$, $i'$ are the obvious inclusions.
By definition of the multiplication in (2.1) we have
$$
\nu_n(\theta_m)=\gamma_G (\Ec_m)=L(i')^*j'_*(\Ec_m),$$ where $\Ec_m$ denotes an
element of $K^{T_s\times G}(T^*_GX')$ whose image by $j'_*$ is equal
to
$$R\psi_* L\phi^*j_*(\Oc_{T^*_GX}(m)).$$
Therefore we have the following formula
\begin{equation}\label{CO:15}\nu_n(\theta_m)=
L(i')^*R\psi_*L\phi^*j_*(\Oc_{T^*_GX}(m)).
\end{equation}
Now, we compute the right hand side of (\ref{CO:15}). Recall that
$Z=G\times_B(\ben^g\times\ben^{g})$ and that
$T^*X'=\gen^g\times\gen^{g}$. For any finite dimensional $T_s\times
B$-module $V$ we'll abbreviate $$\Lambda_{-1}(V)=\sum_{r\geqslant
0}(-1)^{r}\Lambda^r(V)\in R_B.
$$
We have
$$
j_*(\Oc_{T^*_GX}(m))=\Gamma\bigl(\theta_m\otimes\Lambda_{-1}(p\nen)^*\otimes\Oc_{\nen\times
C_\hen}\bigr).
$$
Therefore we have also
$$
L\phi^*
j_*(\Oc_{T^*_GX}(m))=\Gamma\bigl(\theta_m\otimes\Lambda_{-1}(p\nen)^*\otimes\Oc_{\ben^g\times\ben^{g}}\bigr).
$$
Under tensoring by $K_G$ the maps $i_*$, $Li^*$ become invertible by
the Thomason localization theorem. We'll abbreviate $x\ben=\Sum_r
x_r\ben$ and $y\ben=\Sum_ry_r\ben$. We have
$$\aligned
\nu_n(\theta_m)&=L(i')^*R\psi_*i_*\Gamma\bigl(\theta_m\otimes\Lambda_{-1}(p\nen
-x\ben-y\ben)^*\bigr),\cr
&=L(i')^*i'_*R\pi_*\Gamma\bigl(\theta_m\otimes\Lambda_{-1}(p\nen-x\ben-y\ben)^*\bigr),\cr
&=\Lambda_{-1}(x\gen+y\gen)^*\otimes
R\pi_*\Gamma\bigl(\theta_m\otimes\Lambda_{-1}(p\nen-x\ben-y\ben)^*
\bigr),\cr &=
R\pi_*\Gamma\bigl(\theta_m\otimes\Lambda_{-1}(p\nen+x\nen^*+y\nen^*)^*\bigr).
\endaligned$$
Thus the integration over the fixed points subset $(G/B)^{H}$ of $G/B$
yields the desired formula
$$\nu_n(\theta_m)=Sym_n(k(z_1,z_2,...z_n)\theta_m),$$
where $$k(z_1,z_2,\dots z_n)=\prod_{i<j}k(z_i/z_j),$$
\begin{equation}\label{E:defk(z)}
k(z)=(1-z)^{-1}(1-p^{-1}z^{-1})\prod_r(1-x_r^{-1}z)(1-y_r^{-1}z).
\end{equation}
\end{proof}

\vspace{.2in}

\paragraph{\textbf{2.6.}} 
Let $k(z)$ be given by (\ref{E:Kg(z)}). 
Then we have
$$\Ab_{k(z)} 
\subset \bigoplus_{n} R_s[z_1^{\pm 1}, \ldots, z_n^{\pm 1}]^{\mathfrak{S}_n}.$$
Comparing formula (\ref{E:KHallform}) with the definition
of shuffle algebras given in Section~1.9 yields the following corollary to Proposition~\ref{P:KHall}.

\vspace{.1in}

\begin{theo}\label{T:Isoktheory} 
The maps $\bar\Cb_n \to K_{GL_n}, \;x \mapsto Li^*(x)$ give rise to
a graded algebra isomorphism
\begin{equation*}\label{E:isomKHall}
\Phi~: \bar\Cb \stackrel{\sim}{\to} \Ab_{k(z)} 
\end{equation*}
such that
$\Phi( \theta x)=\theta \cdot \Phi(x)$
for $x \in\bar\Cb_n$ and $\theta \in R_{GL_n}$.
\end{theo}

\vspace{.2in}

\vspace{.2in}

\section{The isomorphism}

\vspace{.1in}

In this section we state our main result (whose proof is now obvious) 
relating the 
spherical $K$-theoretic Hall algebra $\bar\Cb$ 
and the (twisted) spherical Hall algebra 
$\dot{\UU}^>_{\E}$ of a smooth projective curve $\E$.
As an application, we  compute the images, under our correspondence, of the skyscraper sheaf
$\mathcal{O}_{\text{triv}}$ at the trivial local system of rank $r$ over $\E$ 
and of the constant function $\mathbf{1}^{\vvec}_{r}$ over 
${Bun}_{GL_r}\E$.

\vspace{.2in}

\paragraph{\textbf{3.1.}}  Recall that we have fixed an identification
\begin{align*}
\mathbf{C}_1=K^{T_s \times \CC^\times}(\CC^{2g})&\stackrel{\sim}{\to} {R}_s[z_1^{\pm 1}],\quad
z^d [\mathcal{O}_{\CC^{2g}}] \mapsto z_1^d,\quad d\in\ZZ.
\end{align*}
We identify the torus $T_a$ arising on the Hall algebra side,
see (\ref{E:torusa}), with the torus $T_s$ arising on the K-theory side,
see (\ref{E:Toruss}), via the map
$$\eta_i \mapsto e_i^{-1},\quad \bar\eta_i \mapsto f_i^{-1},\quad i=1, \ldots, g.$$ This induces a ring isomorphism 
\begin{equation}\label{E:equatevariables}
R_{{a}} \simeq
R_{s},\quad \alpha_i \mapsto x_i^{-1},\quad \bar\alpha_i\mapsto y_i^{-1},\quad q \mapsto p^{-1}.
\end{equation}
From now on we will simply write $R$ for rings $R_a$, $R_s$ and $K$ for the fraction field of $R$. 
Observe that under the identification 
(\ref{E:equatevariables})
we have $\tilde{\zeta}(z)=k(z)$, see Sections~1.9 and ~2.5. 
Note also that we have, in $K^{T_s\times\CC^\times}(\CC^{2g}),$
$$[\mathcal{O}_{\{0\}}]=\prod_{l=1}^g(1-x_l)(1-y_l) [\mathcal{O}_{\CC^{2g}}]=q^{-g}\prod_{l=1}^g (1-\a_l)(1-\bar\a_l) [\mathcal{O}_{\CC^{2g}}].$$
In addition we have, see e.g., \cite[Section~14]{Milne}, 
\begin{equation}\label{E:triv2}
\prod_{l=1}^g(1-\a_l)(1-\bar\a_l)= \# Pic^0(X)(\mathbb{F}_q),
\end{equation}
so that we get
$$[\mathcal{O}_{\{0\}}]=q^{-g}\#Pic^0(X)(\mathbb{F}_q) [\mathcal{O}_{\CC^{2g}}].$$

\vspace{.1in}

\begin{theo}\label{thm:Main} The assignment 
$z^{d}[\mathcal{O}_{\{0\}}]\mapsto\mathbf{1}^{ss}_{1,d}$,
$d\in\ZZ$,
extends to an $R$-algebra anti-isomorphism 
$\Theta_R~: \bar\Cb\to\dot{\UU}^>_{R}$
such that
$\Theta_R(\pi_n(u)x)=u\bullet\Theta_R(x)$ for
$x\in\bar\Cb_n$ and $u\in\UU^0_R$.
\end{theo}
\begin{proof} It easily follows from the definitions of shuffle
algebras that the identity map is an algebra anti-isomorphism 
$Id~:\Ab_{R,\tilde{\zeta}(z)}\stackrel{\sim}{\to}\Ab_{R,\tilde{\zeta}(z^{-1})}$.
We consider the map $Id'$ defined by 
$$Id'_{|\mathbf{A}_n}=\frac{1}{\prod_{l=1}^g (1-x_l)^n(1-y_l)^n} 
Id_{\mathbf{A}_n}=\frac{q^{ng}}{\#(Pic^0(X)(\fq))^{n}} Id_{\mathbf{A}_n}.$$
The theorem is now a consequence of the chain of maps
$$
\xymatrix{\bar\Cb \ar[r]^-{\Phi}& 
\mathbf{A}_{R, \tilde{\zeta}(z)} \ar[r]^-{Id'}& 
\mathbf{A}_{R, \tilde{\zeta}(z^{-1})} & 
\dot{\UU}^>_{R} \ar[l]_-{\dot{\Upsilon}}},$$
see Corollary~\ref{C:finalshuffle} and Theorem~\ref{T:Isoktheory}. 
Note that all these maps are compatible with the relevant $R_{GL_n}$-actions.
\end{proof}

\vspace{.1in}

\noindent Let 
$\Theta_{K} :\bar\Cb_{K}\to\dot{\UU}^>_{K}$ 
be the extension of scalars of $\Theta_R$ from $R$ to $K$.
It is a $K$-algebra isomorphism.

\vspace{.2in}
 
\addtocounter{theo}{1}

\paragraph{\textbf{Remarks \thetheo}} 
$(a)$ The renormalization above is made so that $\Theta_R$ is compatible with the geometric class field theory.
\vspace{.1mm}

$(b)$ The reader might wonder why $\Theta$ is an anti-isomorphism 
rather than an isomorphism. 
This is only a consequence of our convention concerning the order of the multiplication in the Hall algebra,
 which follows the tradition in that field (see e.g., \cite{Ringel}, \cite{SLectures}). Of course, had we considered
 the Hall algebra with the opposite multiplication (as it is done in \cite{Kap}, for instance) we would have 
 obtained an honest isomorphism.

\vspace{.2in}

\paragraph{\textbf{3.2.}} Now we discuss the preimage, under the correspondence 
$\Theta_{R}$, of the constant function $\mathbf{1}^{\vvec}_r$ on the set 
${Bun}_r(\E)$ of all vector bundles of rank $r$ over $\E$. 
The element $$\mathbf{1}^{\vvec}_{r}=\sum_{rank(\mathcal{V})=r} 1_{\mathcal{V}}$$ strictly speaking 
does not belong to $\H_{\E}$ since it is an infinite sum. We may break it up into terms according to the degree
as $\mathbf{1}^{\vvec}_{r}=\sum_d \mathbf{1}^{\vvec}_{r,d}$. Each $\mathbf{1}^{\vvec}_{r,d}$ is still an infinite
 sum, but this sum belongs to the standard completion $\widehat{\H}_{\E}$ of $\H_{\E}$ defined as
$$\widehat{\H}_{\E}=\bigoplus_{r,d} \widehat{\H}_{\E}[r,d], \qquad \widehat{\H}_{\E}[r,d]=
\{f: \mathcal{I}^{vec}_{r,d} \to \CC\}=\prod_{\mathcal{F} \in \mathcal{I}^{vec}_{r,d}} \CC 1_{\mathcal{F}}.$$
Here $\mathcal{I}^{vec}_{r,d}$ is  the set of isomorphism classes of vector bundles over $\E$ of rank $r$ and 
degree $d$. The completion $\widehat{\H}_{\E}$ is still an algebra, see e.g., \cite[Section~2]{BS}.

\vspace{.2in}

\paragraph{\textbf{3.3.}} Let us begin with a heuristic computation. By \cite[Lemma~1.7]{SLectures} we have 
$$\Delta(\mathbf{1}_{\gamma})=
\sum_{\a + \beta=\gamma} v^{\langle \a, \beta \rangle} \mathbf{1}_{\a} \otimes \mathbf{1}_{\beta},
\quad\gamma \in \Z^2. $$ 
Recall that $v=q^{-1/2}$. Iterating this and restricting to vector bundles, 
we get the following expression for the constant term of $\mathbf{1}^{\vvec}_{r,d}$
\begin{equation}\label{E:consterm}
\begin{split}
J_r(\mathbf{1}^{\vvec}_{r,d})&=\sum_{d_1 + \cdots + d_r=d} v^{\sum_{i<j} \langle (1,d_i), (1,d_j)\rangle} \mathbf{1}^{\ss}_{1,d_1} \otimes \cdots \otimes \mathbf{1}^{\ss}_{1,d_r}\\
&= v^{r(r-1)(1-g)/2}\sum_{d_1 + \cdots + d_r=d} v^{(1-r)d_1 + (3-r)d_2 + \cdots + (r-1)d_r}\mathbf{1}^{\ss}_{1,d_1} \otimes \cdots \otimes \mathbf{1}^{\ss}_{1,d_r}.
\end{split}
\end{equation}
Using the identification $\UU^>_{R}[1] ^{\otimes r} \simeq R[z_1^{\pm 1}, \ldots, z_r^{\pm 1}]$ we may write the generating function
\begin{equation}\label{E:consterm2}
\begin{split}
\sum_d J_r(\mathbf{1}^{\vvec}_{r,d})s^d&=v^c \sum_{d_1, \ldots, d_r} v^{(1-r)d_1 + (3-r)d_2 + \cdots + (r-1)d_r}z_1^{d_1} \cdots z_r^{d_r}s^{\sum d_i}\\
&=v^c \big(\sum_{d_1} (v^{1-r}sz_1)^{d_1}\big) \cdots \big( \sum_{d_r} (v^{r-1}sz_r)^{d_r}\big)\\
&=v^c \delta( v^{1-r}sz_1) \cdots \delta(v^{r-1}sz_r)
\end{split}
\end{equation}
where $c=r(r-1)(1-g)/2$ and where $\delta(z)=\sum_{d \in \Z} z^d$. Recall that the map 
$$\dot{\Upsilon}_{R} : \dot{\UU}_{R}^> \to \mathbf{A}_{R, \tilde{\zeta}(z^{-1})}$$ is, in 
degree $r$, the composition of the constant term map $J_r$ with the multiplication by 
$$\tilde{\zeta}(z_1^{-1}, \ldots, z_r^{-1})=\prod_{i<j} (1-z_j/z_i)^{-1}  \prod_{i<j} 
\left\{(1-v^2z_j/z_i) \prod_{l=1}^g (1-\alpha_l z_j/z_i) (1-\bar\a_l z_j/z_i)\right\}.$$ 
Multiplying formally the right hand side of (\ref{E:consterm2}) by $\tilde{\zeta}(z_1^{-1}, \ldots, z_r^{-1})$ yields 
$$ \Theta_{R}^{-1}(\mathbf{1}^{\vvec}_{r,d})=v^c \delta( v^{1-r}sz_1) \cdots \delta(v^{r-1}sz_r) \tilde{\zeta}(z_1^{-1}, \ldots, z_r^{-1})=v^c \tilde{\zeta}(v^{1-r}, v^{3-r}, \ldots, v^{r-1})=0.$$ 
Thus one would be tempted to directly conclude that $\Theta_{R}^{-1}(\mathbf{1}^{\vvec}_{r,d})=0$ and hence
 $\Theta_{R}^{-1}(\mathbf{1}^{\vvec}_{r})=0$. Of course the above computation is not valid as such since it involves divergent sums.

\vspace{.2in}

\paragraph{\textbf{3.4.}} In order to still make sense of $\Theta_{R}^{-1}(\mathbf{1}_r)$ we will approximate 
each $\mathbf{1}^{\textbf{vec}}_{r,d}$ by a sequence of genuine elements of $\H_{\E}$. For this we will use the 
Harder-Narasimhan filtration on coherent sheaves over $\E$. We refer the reader to \cite{Ssemistables} for the 
precise definitions, and for some of the results stated below. We denote by 
$HN(\mathcal{F})=(\a_1, \a_2, \ldots, \a_s)$ the Harder-Narasimhan type of a 
coherent sheaf $\mathcal{F}$. Recall that
$$\gathered
\a_1, \a_2, \dots, \a_s\in\ZZ^{2,+},\quad
\mu(\a_1)<\mu(\a_2)<\cdots<\mu(\a_s),\cr
\ZZ^{2,+}=\{(r,d)\in\ZZ^2;r\geqslant 1,\ \text{or}\ r=0,d\geqslant 1\},\quad
\mu(r,d)=d/r\in\QQ\cup\{\infty\}.
\endgathered$$
We write $\mathbf{1}^{\ss}_{r,d}$ for the characteristic function of the set of semistable coherent sheaves of rank
$r$ and degree $d$. We have
\begin{equation}\label{E:hopla}
\mathbf{1}^{\textbf{vec}}_{r,d}=\sum_{\underline{\a} \in Y_{\a}} 
v^{\sum_{i < j} \langle \a_i, \a_j \rangle} \mathbf{1}^{\ss}_{\a_1} 
\cdots \mathbf{1}^{\ss}_{\a_s},
\end{equation} 
where $Y_{r,d}$ is the set of all Harder-Narasimhan types 
$\underline{\a}=(\a_1, \ldots, \a_s)$ of weight $\sum_i\a_i=(r,d)$ 
for which $\mu(\a_s) < \infty$. By \cite[Theorem~2.4]{Ssemistables} each
$\mathbf{1}^{\ss}_{\beta}$ belongs to $\UU^>_{\E}$.
So $\mathbf{1}_{r,d}$ is in the completion 
$\widehat{\UU}_{\E}^>$ of $\UU^>_{\E}$. We approximate $\mathbf{1}^{\textbf{vec}}_{r,d}$ 
by partial sums of (\ref{E:hopla}). For any finite subset $Z$ of $Y_{r,d}$, set 
\begin{equation}\label{E:hopla2}
\mathbf{1}^{Z}_{r,d}=\sum_{\underline{\a} \in Z} v^{\sum_{i < j} \langle \a_i, \a_j \rangle} 
\mathbf{1}^{\ss}_{\a_1} \cdots \mathbf{1}^{\ss}_{\a_l}.
\end{equation} 
We consider the following notion of convergence for a sequence of elements of 
$\bar\Cb$. 
Define a $\Z$-grading on the ring $R$ by setting 
$deg(x_i)=deg(y_i)=1$ for $i=1, \ldots, g$. The relation $x_iy_i=p$ 
in $R_s=R$ and 
the assignment $q\mapsto p^{-1}$ in
(\ref{E:equatevariables}) imply that $\deg(v)=1$.
Write 
$R=\bigoplus_l R_l$ for the decomposition into graded pieces. We consider convergence with respect to the 
adic topology induced by this degree. More precisely, let us write 
$$R[z_1^{\pm 1}, \ldots, z_r^{\pm 1}]=\bigoplus_{l} R_l [z_1^{\pm 1}, \ldots, z_r^{\pm 1}],\quad
R_{\geqslant l}[z_1^{\pm 1}, \ldots, z_r^{\pm 1}]=\bigoplus_{l' \geqslant l} R_{l'}[z_1^{\pm 1}, \ldots, z_r^{\pm 1}]. $$
Then a sequence $(u_i)_{i \in I}$ of elements of
$\Vb\subset R[z_1^{\pm 1}, \ldots, z_r^{\pm 1}]$ 
\textit{converges} to $u$ if for any $l$ there exists 
$i_0 \in I$ such that $u-u_i \in R_{\geqslant l}[z_1^{\pm 1}, \ldots, z_r^{\pm 1}]$ for any $i > i_0$.

\vspace{.1in}

\begin{prop}\label{P:buntriv} For any $(r,d)$, the sequence $\Theta_{R}^{-1}(\mathbf{1}^Z_{r,d})$ converges 
to zero as $Z$ tends to $Y_{r,d}$.
\end{prop}
\begin{proof} Let us fix a pair $(r,d)$. 
The argument hinges on the following two lemmas, which are simple 
consequences of Corollary~\ref{C:kloop}. 
Set $R_{[-n,n]}=\bigoplus_{l=-n}^n R_l$.  
Denote by $W \subset \UU^>_{R}$ the $\CC$-subalgebra generated by $\{\mathbf{1}^{\ss}_{1,d}; d \in \Z\}$.

\vspace{.1in}

\begin{lem}\label{L:proof1} 
(a) There exists $n \in \N$ such that for any $\underline{d}=(d_1, \ldots, d_r) \in \Z^r$ we have
$$\Theta_{R}^{-1}(\mathbf{1}^{\ss}_{1,d_1} \cdots \mathbf{1}^{\ss}_{1,d_r}) \in R_{[-n,n]}[z_1^{\pm 1}, \ldots, z_r^{\pm 1}].$$

(b) There exists $n' \in \N$ such that for 
$r' \leqslant r$ and $d' \in \Z$ we have 
$$\mathbf{1}^{\ss}_{r',d'} \in R_{-[n',n']}W.$$
\end{lem}

\vspace{.1in}

\noindent
By Lemma~\ref{L:proof1}$(b)$ we have, for any HN type 
$\underline{\a}=(\a_1, \ldots, \a_s)$ of weight $(r,d)$
$$\mathbf{1}^{\ss}_{\a_1} \cdots \mathbf{1}^{\ss}_{\a_s} \in R_{[-sn',sn']}W$$
and using Lemma~\ref{L:proof1}$(a)$ we get 
$$\Theta_R^{-1}(\mathbf{1}^{\ss}_{\a_1} \cdots \mathbf{1}^{\ss}_{\a_s}) 
\in R_{-[sn'-n,sn'+n]}[z_1^{\pm 1}, \ldots, z_r^{\pm 1}].$$
From (\ref{E:hopla}) and from the fact that 
$\sum_{i<j} \langle \a_i, \a_j \rangle \to \infty$ as the HN type 
$\underline{\a}$ goes to infinity, we deduce that the sequence 
$\Theta_{R}^{-1}(\mathbf{1}^Z_{r,d})$ indeed converges in 
$R[z_1^{\pm 1}, \ldots, z_r^{\pm 1}]$ as $Z$ tends to $Y_{r,d}$. 
One shows, by the same calculation as in (\ref{E:consterm2}) that the limit 
is equal to zero. We leave the details to the reader. \end{proof}

\vspace{.2in}
 
\addtocounter{theo}{1}

\paragraph{\textbf{Remark \thetheo}} Proposition~\ref{P:buntriv} says that in 
any lift of the isomorphism $\Theta_R$ to an equivalence of triangulated 
categories, the constant sheaf $\overline{\mathbb{Q}_l}_{{Bun}_{GL_r}X}$ 
would be mapped to a complex of coherent sheaves on the stack 
${Loc}_{GL_r}X$ whose class in the Grothendieck group is zero. As explained 
to us by V. Lafforgue, this is indeed expected of the geometric Langlands 
correspondence~: the
constant sheaf $\overline{\mathbb{Q}_l}_{{Bun}_{GL_r}X}$ should be mapped 
to some unbounded acyclic complex in $D(Coh({Loc}_{GL_r}X))$.

\vspace{.2in}

\paragraph{\textbf{3.5.}} Let us now fix some $r \geqslant 1$ and describe 
the image under our correspondence of the class 
$[\mathcal{O}_{\text{triv}_r}]=[\mathcal{O}_{\{0\}}]$
in $K^{T_s \times GL_r}(C_{\mathfrak{gl}_r})$. Write 
$x\mathfrak{g}=\sum x_l \mathfrak{g}$ and $y\mathfrak{g}=\sum y_l \mathfrak{g}$.
Let $i: \{0\} \to {C}_{\mathfrak{gl}_r}$ be the inclusion.
We have
\begin{equation*}
Li^* ([\mathcal{O}_{\{0\}}])=
\Lambda_{-1}(x\mathfrak{g} + y\mathfrak{g})=\prod_{i,j=1}^r \kappa(z_i/z_j)
=\prod_{l=1}^g (1-x_l)^r(1-y_l)^r \cdot \prod_{i\neq j} \kappa(z_i/z_j)
\end{equation*}
where 
$\kappa(z)=\prod_{l=1}^g (1-x_lz)(1-y_lz)$. Next, recall the weighted 
symmetrization map $\Psi_r$ used in the definition of 
$\mathbf{A}_{\tilde{\zeta}(z)}$ in Section 1.10.
Here $\tilde{\zeta}(z)$ is given by (\ref{E:deftzetax(z)}). 
A direct computation yields 
\begin{equation}\label{E:triv1}
\begin{split}
\Psi_r\bigg(z_1^{r-1}z_2^{r-3} \cdots z_r^{1-r}& \cdot 
\prod_{i<j}\prod_{l=1}^g (1-x_l^{-1}z_j/z_i)(1-y_l^{-1}z_j/z_i)\bigg)=\\
&=(-1)^{r(r-1)/2}[r]! \prod_{i \neq j} \prod_{l=1}^g (1-x_l^{-1}z_i/z_j)(1-y_l^{-1}z_i/z_j)\\
&=(-1)^{r(r-1)/2}[r]! q^{gr(r-1)/2}\prod_{i \neq j} \kappa(z_i/z_j)
\end{split}
\end{equation} 
which, up to a non-zero factor in $K$, 
is equal to $Li^*([\mathcal{O}_{\{0\}}])$. 
Here we have set 
$$[s]=1+ q+ \cdots + q^{s-1},\quad [s]!=[1] [2] \cdots [s].$$
This shows in particular that $[\mathcal{O}_{\{0\}}]$ belongs to the subalgebra
$\bar\Cb_{K}$ of 
$K^{T_s \times GL_r}({C}_{\mathfrak{gl}_r})\otimes_R K$,
and yields an
expression for $\Theta_K([\mathcal{O}_{\{0\}}])$. In order to write this 
expression in a nice way, we introduce the following notation. 
Define a ${K}$-linear map
\begin{equation*}
Ind~: {K}[z_1^{\pm 1}, \ldots, z_r^{\pm 1}] \to \dot{\UU}^>_{K},\quad 
z_1^{d_1} \cdots z_r^{d_r} \mapsto \mathbf{1}^{\ss}_{1,d_1} \cdots 
\mathbf{1}^{\ss}_{d_r}.
\end{equation*}
If $\sigma=(\sigma_1, \ldots, \sigma_r) \in \Z^r$ 
we write $z^\sigma=z_1^{\sigma_1} \cdots z_r^{\sigma_r}$.  
We have the following formula. 

\vspace{.1in}

\begin{prop}\label{P:cohtriv} For $r\geqslant 1$ we have
\begin{equation}\label{E:cohtriv}
\Theta_K([\mathcal{O}_{\mathrm{triv}_r}])=(-1)^{r(r-1)/2}
\frac{q^{-gr(r+1)/2}}{[r]!} Ind \bigg(z^{-2\rho} 
\prod_{\substack{\sigma \in \Delta^+ }}
\prod_{l=1}^g (1-\a_l z^\sigma)(1-\bar\a_l z^{\sigma})\bigg)
\end{equation}
where $\Delta^+ \subset \Z^r$ is the set of positive roots of $\mathfrak{gl}_r$.
\end{prop}
\begin{proof} This is a direct consequence of (\ref{E:triv1}) 
and Theorem~\ref{thm:Main}.
\end{proof}

\vspace{.2in}
 
\addtocounter{theo}{1}

\paragraph{\textbf{Remark \thetheo}} Proposition~\ref{P:cohtriv} is stated in 
terms of the multiplication in the twisted spherical Hall algebra 
$\dot{\UU}^>_{K}$. If instead one uses the usual spherical Hall algebra 
$\UU^>_{K}$ then the expression for $\Theta_K([\mathcal{O}_{\mathrm{triv}_r}])$
is a little bit more symmetric
\begin{equation}\label{E:cohtriv2}
\Theta_K([\mathcal{O}_{\mathrm{triv}_r}])=(-1)^{r(r-1)/2}
\frac{q^{-gr(r+1)/2}}{[r]!} Ind \bigg(z^{-2g\rho} 
\prod_{\substack{\sigma \in \Delta^+ }}
\prod_{l=1}^g (1-\a_l z^\sigma)(1-\bar\a_l z^{\sigma})\bigg).
\end{equation}

\vspace{.2in}
 
\addtocounter{theo}{1}

\paragraph{\textbf{Example \thetheo}} Let us assume that $X$ is an elliptic 
curve and that $r=2$. A direct computation shows that, up to a global factor, 
the function
$\Theta_K([\mathcal{O}_{\text{triv}_2}])$ takes the following non-zero values on the 
closed points of ${Bun}_{2,0}X$
$$
\Theta_K([\mathcal{O}_{\text{triv}_2}]) (\mathcal{L}_0 \oplus \mathcal{L}_0')= 
(1+q)(\#X(\fq)-2)$$
if $\mathcal{L}_0, \mathcal{L}_0' \in Pic^0X, \;\mathcal{L}_0 \neq 
\mathcal{L}_0'$, 
$$\Theta_K([\mathcal{O}_{\text{triv}_2}]) (\mathcal{L}_0^{(2)})= 
(1+q)(\#X(\fq)-1)$$
if $\mathcal{L}^{(2)}_0$ is the (unique) indecomposable self-extension of 
some $\mathcal{L}_0 \in Pic^0X$,
$$\Theta_K([\mathcal{O}_{\text{triv}_2}]) (\mathcal{V})= (1+q)\#X(\fq)$$
if $\mathcal{V}$ is a rank two stable bundle, and finally
$$\Theta_K([\mathcal{O}_{\text{triv}_2}]) 
(\mathcal{L}_{-1} \oplus \mathcal{L}_1)= q$$
for $\mathcal{L}_{-1} \in Pic^{-1}X,$ $ \mathcal{L}_1 \in Pic^1X$.
Thus $\Theta_K([\mathcal{O}_{\text{triv}_2}])$ is supported on 
the union of the two Harder-Narasimhan stratas $S_{2,0}$ and 
$S_{(1,-1), (1,1)}$. 
It is easy to see, using e.g., (\ref{E:cohtriv2}), that for a curve $X$ of genus
$g$ the function $\Theta_K([\mathcal{O}_{\text{triv}_2}])$ is supported on the 
union of Harder-Narasimhan stratas 
$$S_{(2,0)} \cup S_{(1,-1),(1,1)} \cup \cdots \cup S_{(1,-g),(1,g)}. $$
More generally, if $\mathfrak{n}_+ \subset \mathfrak{gl}_r$ is 
the positive nilpotent subalgebra and if $\Phi$ is the set of 
weights occuring in $\Lambda^\bullet (H^1(X,\qlb) \otimes \mathfrak{n}_+)$ then
$\Theta_K([\mathcal{O}_{\text{triv}_r}])$ is supported on the union of 
Harder-Narasimhan strata whose type belongs to the convex hull of 
$\{\alpha-2g\rho; \alpha \in \Phi\}$.

\vspace{.2in}

\centerline{\textbf{Acknowledgements.}}

\vspace{.1in}

We would like to thank V. Lafforgue for useful conversations.

\vspace{.2in}

\centerline{APPENDICES}

\appendix

\vspace{.2in}

\section{The principal Hall algebra}

\vspace{.15in}

To our knowledge, the Hall algebra $\H_\E$ cannot be regarded as a shuffle algebra.
However, there is a subalgebra  of $\H_\E$ which strictly contains ${\UU}_{\E}$ and which admits a similar
description. It is the principal Hall algebra that we describe now.

\vspace{.15in}

\paragraph{\textbf{A.1.}} We define the \textit{principal Hall algebra} 
$\H^{pr}_\E$ of $X$ as the subalgebra of $\H_{\E}$ which is generated by 
$\H_{\E}[0]$ and $\H_{\E}[1]$, i.e., as the subalgebra of $\H_{\E}$ generated 
by all the characteristic functions $1_{\mathcal{F}}$ for $\mathcal{F}$ a 
torsion sheaf or a line bundle. We have 
$$\UU_{\E} \subset \H^{pr}_{\E} \subset \H_{\E},$$ and unless 
$\E \simeq \mathbb{P}^1$ the inclusions are strict. Let $\H^{>,pr}_{\E}$ be 
the subalgebra generated by the functions $1_{\mathcal{L}}$ for 
$\mathcal{L} \in Pic(\E)$.
As for the spherical Hall algebra, we have 
$$\H^{pr}_{\E} \simeq \H^{>,pr}_{\E} \otimes \H_{\E}[0].$$ The aim of this 
section is to give a realization of $\H^{>,pr}_{\E}$ as a shuffle algebra.

\vspace{.2in}

\paragraph{\textbf{A.2.}} Let $\widehat{Pic(\E)}$ be the group of all complex 
characters $\chi~: Pic(\E) \to \CC^\times$. 
The group $\widehat{Pic(\E)}$ fits in a 
natural sequence
$$\xymatrix{ 1 \ar[r] & \CC^\times\ar[r]^-{\iota} & 
\widehat{Pic(\E)} \ar[r]^-{r} 
& \widehat{Pic^0(\E)} \ar[r] & 1}$$
with 
\begin{equation}\label{iota}\iota: \CC^\times\to \widehat{Pic(\E)}, \quad
z \mapsto (\mathcal{L} \mapsto z^{deg(\mathcal{L})}),\end{equation} and with 
$r:\widehat{Pic(\E)} \to \widehat{Pic^0(\E)}$ being the restriction. 
We will write $\rho \sim \chi$ for two characters satisfying
$r(\rho)=r(\chi)$.
For $d \in \Z$ and $\chi : Pic(\E) \to \CC^\times$ we set
$$\mathbf{1}^{\chi}_{1,d}=\sum_{\mathcal{L} \in Pic^d(\E)} 
\chi(\mathcal{L}) 1_{\mathcal{L}}.$$
When $\chi=1$, the trivial character, we recover the function 
$\mathbf{1}^{\ss}_{1,d}$ introduced in Section~1.6.
We need to introduce certain elements of $\H_{Tor(\E)}$. The determinant 
$Vec(\E) \to Pic(\E)$ factors to a morphism of abelian groups 
$K_0(\E) \to Pic(\E)$. This allows to make sense of the determinant 
$det(\mathcal{T}) \in Pic^d(\E)$ of a torsion sheaf $\mathcal{T}$ of degree $d$.
Recall from Section~1.3 the elements $\mathbf{1}_{0,l}, T_{0,l}$ and 
$\theta_{0,l}$ in $H_{Tor(\E)}$. For $\chi \in \widehat{Pic(\E)}$ and 
$l \geqslant 0$ we set
$$\gathered
\mathbf{1}^{\chi}_{0,d}=\sum_{\substack{\mathcal{T} \in Tor(\E)\\ 
deg(\mathcal{T})=d}} \chi(det(\mathcal{T})) 1_{\mathcal{T}},\cr
T^{\chi}_{0,d} =\sum_{\mathcal{T} \in Tor(\E)} \chi(det(\mathcal{T})) 
T_{0,d}(\mathcal{T}) 1_{\mathcal{T}},\cr
\theta^{\chi}_{0,l}=\sum_{\mathcal{T} \in Tor(\E)} 
\chi(det(\mathcal{T})) \theta_{0,l}(\mathcal{T}) 1_{\mathcal{T}}.
\endgathered$$
Using the additivity of the determinant one easily checks that as before
\begin{equation}\label{E:thetaprin}
1+ \sum_{d \geqslant 1} \mathbf{1}^{\chi}_{0,d} s^d=\exp \bigg( \sum_{d}
\frac{T^{\chi}_{0,d}}{[d]}s^d\bigg),\quad 1+ \sum_{d \geqslant 1} 
\theta^{\chi}_{0,d}
s^d=\exp \bigg( (v^{-1}-v)\sum_{d} T^{\chi}_{0,d}s^d\bigg).
\end{equation}

\vspace{.2in}

\paragraph{\textbf{A.3.}} The elements 
$\mathbf{1}^{\chi}_{1,n}, \theta^{\chi}_{0,d}, etc,$ introduced above satisfy 
properties very similar to those of the elements 
$\mathbf{1}^{\ss}_{1,n}, \theta_{0,d}, etc.$ 
We summarize these properties in the next few lemmas.

\vspace{.1in}

\begin{lem}\label{L:printheta} The following holds 
\begin{enumerate}
\item[$(a)$] $\widetilde{\Delta}(\theta^{\chi}_{0,d}) =\sum_{l=0}^d \theta_{0,l}^{\chi} \boldsymbol{\kappa}_{0,d-l} \otimes \theta_{0,d-l}^{\chi}$,
\item[$(b)$] $\widetilde{\Delta}(T^{\chi}_{0,d})=
T^{\chi}_{0,d} \otimes 1 + \boldsymbol{\kappa}_{0,d} \otimes T^{\chi}_{0,d}.$
\end{enumerate}
\end{lem}
\begin{proof}
We prove $(a)$. From Lemma~\ref{L:Hall1} ($b$) we get 
$$\widetilde{\Delta}(\theta^{\mathcal{K}}_{0,d})=\sum_{l=0}^d \sum_{\mathcal{L}, \mathcal{L}'}\theta^{1, \mathcal{L}}_{0,l} \boldsymbol{\kappa}_{0,d-l} \otimes \theta^{1, \mathcal{L}'}_{0,d-l}$$
where the sum ranges over the pairs $\mathcal{L} \in Pic^d(\E), \mathcal{L}' \in Pic^{d-l}(\E)$ satisfying $\mathcal{L} + \mathcal{L}'=\mathcal{K}$. Thus
\begin{equation*}
\begin{split}
\widetilde{\Delta}(\theta^{\chi,\mathcal{K}}_{0,d})=\chi(\mathcal{K})\widetilde{\Delta}(\theta^{\chi,\mathcal{K}}_{0,d})&=
\sum_{l=0}^d \sum_{\mathcal{L}, \mathcal{L}'}\chi(\mathcal{L})\theta^{1, \mathcal{L}}_{0,l} \boldsymbol{\kappa}_{0,d-l} \otimes \chi(\mathcal{L}')\theta^{1, \mathcal{L}'}_{0,d-l}\\
&=\sum_{l=0}^d \sum_{\mathcal{L}, \mathcal{L}'}\theta^{\chi, \mathcal{L}}_{0,l} \boldsymbol{\kappa}_{0,d-l} \otimes \theta^{\chi, \mathcal{L}'}_{0,d-l}.
\end{split}
\end{equation*}
Summing over all $\mathcal{K}$ yields $(a)$. The proof of ($b$) is entirely similar.
\end{proof}

\vspace{.1in}

\begin{lem}\label{L:printor}
For any $\chi \in \widehat{Pic(\E)}$ we have
\begin{equation}\label{E:princop}
\widetilde{\Delta}(\mathbf{1}^{\chi}_{1,d})=
\mathbf{1}^{\chi}_{1,d} \otimes 1 + 
\sum_{l \geqslant 0} \theta^{\chi}_{0,l}\boldsymbol{\kappa}_{1,d-l} 
\otimes \mathbf{1}^{\chi}_{1,d-l}.
\end{equation}
\end{lem}
\begin{proof} For any $\mathcal{K} \in Pic^l(\E)$ let us denote by 
$\theta_{0,l}^{\chi, \mathcal{K}}$ the projection of $\theta_{0,l}^{\chi}$ 
to the subset of torsion sheaves of determinant $\mathcal{K}$. 
From (\ref{E:Hall1}) we immediately deduce, for any fixed line bundle 
$\mathcal{L} \in Pic^d(\E)$
$$\widetilde{\Delta}(1_{\mathcal{L}})=1_{\mathcal{L}} \otimes 1 + 
\sum_{l \geqslant 0} \theta^{1,\mathcal{L}-\mathcal{L}'}_{0,l}
\boldsymbol{\kappa}_{1,d-l} \otimes 1_{\mathcal{L}'}.$$
Summing over $\mathcal{L}$ this yields
\begin{equation*}
\begin{split}
\widetilde{\Delta}(\mathbf{1}^{\chi}_{1,d})&=\mathbf{1}^{\chi}_{1,d} \otimes 
1 + \sum_{l \geqslant 0} \sum_{\substack{\mathcal{L} \in Pic^d(\E) \\ 
\mathcal{L}' \in Pic^{d-l}(\E)}}\chi(\mathcal{L})\theta^{1,\mathcal{L}-
\mathcal{L}'}_{0,l}\boldsymbol{\kappa}_{1,d-l} \otimes 1_{\mathcal{L}'}\\
&=\mathbf{1}^{\chi}_{1,d} \otimes 1 + \sum_{l \geqslant 0} 
\sum_{\substack{\mathcal{L} \in Pic^d(\E) \\ \mathcal{L}' \in 
Pic^{d-l}(\E)}}\chi(\mathcal{L}-\mathcal{L}')\chi(\mathcal{L}')
\theta^{1,\mathcal{L}-\mathcal{L}'}_{0,l}\boldsymbol{\kappa}_{1,d-l} \otimes 
1_{\mathcal{L}'}\\
&=\mathbf{1}^{\chi}_{1,d} \otimes 1 + \sum_{l \geqslant 0} 
\sum_{\substack{\mathcal{L} \in Pic^d(\E) \\ \mathcal{L}' \in Pic^{d-l}(\E)}}
\theta^{\chi,\mathcal{L}-\mathcal{L}'}_{0,l}\boldsymbol{\kappa}_{1,d-l} \otimes 
\chi(\mathcal{L}')1_{\mathcal{L}'}\\
&=\mathbf{1}^{\chi}_{1,d} \otimes 1 + 
\sum_{l \geqslant 0} \theta^{\chi}_{0,l}\boldsymbol{\kappa}_{1,d-l} 
\otimes \mathbf{1}^{\chi}_{1,d-l}.
\end{split}
\end{equation*}
\end{proof}

\vspace{.1in}

\noindent Consider the zeta function of $\chi$ defined as
$$\zeta^{\chi}_{\E}(z)=\exp \bigg(\sum_{d \geqslant 1} 
\sum_{\substack{l |d \\ x \in Y_l}} \chi(det(\mathcal{O}_x)) 
\frac{l}{d} z^d\bigg),$$
where $Y_{l}$ stands for the set of closed points of $\E$ of degree $l$. 
When $\chi=1$ we have of course $\zeta^{\chi}_{\E}=\zeta_{\E}$.
It is known that when $\chi \not\sim 1$ the function $\zeta^{\chi}_{\E}$ is a 
polynomial of degree $2g-2$.

\vspace{.1in}

\begin{lem}\label{L:prinscal} We have
\begin{enumerate}
\item[($a$)] $( \mathbf{1}^{\chi}_{1,n}, \mathbf{1}^{\rho}_{1,n})_G =0$ unless $\rho \sim \chi$, and $\langle \mathbf{1}^{\chi}_{1,n}, \mathbf{1}^{\chi}_{1,n}\rangle=\#Pic^0(\E)/(q-1)$.
\item[($b$)] $(T^{\chi}_{0,d}, T^{\rho}_{0,d})_G=
\frac{v^d[d]}{v^{-1}-v} \sum_{l |d}\sum_{x \in Y_d}\chi 
\overline{\rho}(det(\mathcal{O}_x)) \frac{l}{d}$.
\end{enumerate}
\end{lem}
\begin{proof} Statement ($a$) follows directly from the definitions and from the orthogonality property of characters. The proof of statement ($b$) is an easy modification of Lemma~\ref{L:Hall1} ($c$).
\end{proof}

\vspace{.1in}

\noindent
Recall the Hecke action, see Section~1.5,
$$\bullet~: \H_{Tor(\E)} \otimes \H_{Vec(\E)} \to 
\H_{Vec(\E)}, \quad f \otimes g \mapsto \omega(f \cdot g).$$

\vspace{.1in}

\begin{lem}\label{L:prinhecke} There are complex numbers $\xi_d^{\sigma}$ for 
$\sigma \in \widehat{Pic(\E)}$ and $d \geqslant n$ such that for any two characters 
$\chi, \rho$ and any $d \geqslant 0$ we have
$$\theta^{\chi}_{0,d} \bullet \mathbf{1}^{\rho}_{1,n}=
\xi_{d}^{\chi\overline{\rho}} \mathbf{1}^{\rho}_{1,n+d}.$$
Moreover we have
\begin{equation}\label{E:heckethetaprin}
\theta^{\chi}_{0,d} \mathbf{1}^{\rho}_{1,n}=
\sum_{l=0}^d \zeta^{\chi\overline{\rho}}_{l} 
\mathbf{1}^{\rho}_{1,n+l}\theta^{\chi}_{0,d-l}
\end{equation}
and, as a series in $\CC[[s]]$,
\begin{equation}\label{E:heckethetaprin2}
\sum_{d \geqslant 0} \xi^{\chi}_{d}s^d=
\frac{\zeta^{\chi}_{\E}(s)}{\zeta^{\chi}(q^{-1}s)}.
\end{equation} 
\end{lem}
\begin{proof} Since $T^{\chi}_{0,d}$ is primitive,
i.e., $\Delta(T^{\chi}_{0,d})=T^{\chi}_{0,d} \otimes 1 + 
1 \otimes T^{\chi}_{0,d}$, we have
$$[T^{\chi}_{0,d}, \mathbf{1}^{\rho}_{1,n}]=
T^{\chi}_{0,d} \bullet \mathbf{1}^{\rho}_{1,n} \in \H_{Vec(X)}[1,n+d]=
\bigoplus_{\sigma} \CC \mathbf{1}^{\sigma}_{1,n+d},$$ where the sum ranges over 
a complete set of inequivalent characters $\sigma \in \widehat{Pic(\E)}$. 
Let us write $[T^{\chi}_{0,d}, \mathbf{1}^{\rho}_{1,n}]=
\sum_{\sigma} u^{\sigma}_{\chi,\rho} \mathbf{1}^{\sigma}_{1,n+d}$ 
for some scalars $u^{\sigma}_{\chi,\rho}$. We will prove that 
$u^{\sigma}_{\chi,\rho}=0$ unless $\sigma$ is equivalent to $\rho$. 
Indeed, if $\sigma \not\sim \rho$ then
$$([T^{\chi}_{0,d}, \mathbf{1}^{\rho}_{1,n}], 
\mathbf{1}^{\sigma}_{1,n+d})_G=u_{\chi,\rho}^{\sigma} 
(\mathbf{1}^{\sigma}_{1,n+d}, \mathbf{1}^{\sigma}_{1,n+d})_G$$
while on the other hand by Lemma~\ref{L:prinscal} ($a$)
\begin{equation*}
\begin{split}
([T^{\chi}_{0,d}, \mathbf{1}^{\rho}_{1,n}],\mathbf{1}^{\sigma}_{1,n+d})_G&
=(T^{\chi}_{0,d} \mathbf{1}^{\rho}_{1,n},\mathbf{1}^{\sigma}_{1,n+d})_G\\
&=(T^{\chi}_{0,d} \otimes \mathbf{1}^{\rho}_{1,n}, 
\widetilde{\Delta}(\mathbf{1}^{\sigma}_{1,n+d}))_G
=(T^{\chi}_{0,d}, \theta^{\sigma}_{0,d})_G (\mathbf{1}^{\rho}_{1,n}, 
\mathbf{1}^{\sigma}_{1,n})_G=0.
\end{split}
\end{equation*}
It follows that $[T^{\chi}_{0,d}, \mathbf{1}^{\rho}_{1,n}] \in 
\CC \mathbf{1}^{\rho}_{1,n+d}$. Using (\ref{E:thetaprin}) we deduce 
that $\theta^{\chi}_{0,d} \bullet \mathbf{1}^{\rho}_{1,n} \in 
\CC \mathbf{1}^{\rho}_{1,n+d}$ as wanted. Statement (\ref{E:heckethetaprin}) 
is a consequence of Lemma~\ref{L:printheta} ($a$). To prove 
(\ref{E:heckethetaprin2}), note that by the above calculation together with 
Lemma~\ref{L:prinscal} ($b$) we have
$$T^{\chi}_{0,d} \bullet \mathbf{1}^{\rho}_{1,n}=
v^d[d] \sum_{l |d}\sum_{x \in Y_d} \chi \overline{\rho}
(det(\mathcal{O}_x)) \frac{l}{d} \cdot \mathbf{1}^{\rho}_{1,n+d},$$
from which it entails
\begin{equation*}
\begin{split}
\sum_{d \geqslant 0} \xi^{\chi \overline{\rho}}_{d}s^d&=
\exp \bigg( (v^{-1}-v) \sum_{d \geqslant 1} v^d[d] 
\sum_{l |d}\sum_{x \in Y_d} \chi 
\overline{\rho}(det(\mathcal{O}_x)) \frac{l}{d} s^d\bigg),\\
&=\exp \bigg( \sum_{d \geqslant 1} (1-q^{-d}) 
\sum_{l |d}\sum_{x \in Y_d} 
\chi \overline{\rho}(det(\mathcal{O}_x)) \frac{l}{d} s^d\bigg),\\
&=
\frac{\zeta^{\chi\overline{\rho}}_{\E}(s)}
{\zeta^{\chi\overline{\rho}}(q^{-1}s)}.
\end{split}
\end{equation*}
Lemma~\ref{L:prinhecke} is proved.
\end{proof}

\vspace{.2in}

\paragraph{\textbf{A.4.}} We may use the constant term map and the above 
Lemmas~\ref{L:printheta}-\ref{L:prinhecke} to get a combinatorial
realization of $\H^{>,pr}_{\E}$ as a shuffle algebra. 
We describe this realization below and leave the details of the proof 
to the reader.
Let $\Xi$ be a set of representatives in $\widehat{Pic(\E)}$ for the 
equivalence classes $\widehat{Pic(\E)}/Im(\iota)$, where $\iota$ is as in
(\ref{iota}). We may (and we will) 
choose the elements of $\Xi$ to be unitary, i.e., $|\chi(\mathcal{L})|=1$ 
for any $\chi \in \Xi$ and $\mathcal{L} \in Pic(\E)$. Since the characters of a
finite abelian group $B$ form a $\CC$-basis of the set of all functions 
$B \to \CC$, the collection of elements 
$\{\mathbf{1}^{\chi}_{1,d};d \in \Z, \;\chi \in \Xi\}$ forms a basis of 
$\H^>_{\E}[1]$. Hence it generates $\H^{>,pr}_{\E}$. 
For $\chi_1,\dots,\chi_r \in \Xi$ let 
$\delta_{\chi_1, \ldots, \chi_r} \in \CC[\Xi^r]$ be the characteristic function
of $\{(\chi_1, \ldots, \chi_r)\}$. This allows us to write
$$\D_r=\CC[\Xi^r\times(\CC^\times)^r] = \bigoplus_{\chi_1, \ldots, \chi_r}\bigoplus_{d_1, \ldots, d_r} 
\CC \delta_{\chi_1, \ldots, \chi_r} x_1^{d_1} \cdots x_r^{d_r}, \quad 
\chi_i \in \Xi,\quad d_i \in \Z.$$
To any $\chi \in \Xi$ we associate a function $g^{\chi}(z) \in \CC(z)$ 
as follows. Let $\{\beta_1^{\chi}, \ldots, \beta_{2g}^{\chi}\}$ be the 
Frobenius eigenvalues in $H^1(\E \otimes \fqb, \mathcal{L}_{\chi})$ where 
$\mathcal{L}_{\chi}$ is the rank one local system on $\E \otimes \fqb$ 
associated with $\chi$. We have 
$$\zeta_{\E}^{\chi}(z)=\prod_{i=1}^{2g}(1-\beta_i^{\chi}z)/(1-z)(1-qz),\quad
|\b_i^{\chi}|=q^{1/2},
$$
see e.g., \cite{De}. We define
$$\gamma_{\chi}=q^{-1}\prod_{i=1}^{2g} \beta^{\chi}_i=
\prod_{l=0}^2 det (Fr; H^i(\E \otimes \fqb, \mathcal{L}_{\chi}))^{(-1)^{i+1}}.$$
This factor enters into the functional equation
\begin{equation}
\zeta^{\chi}(q^{-1}z)=\gamma_{\chi}^{-1} z^{2(g-1)} 
\zeta^{\overline{\chi}}(z^{-1}).
\end{equation}
Put 
$$\gathered
g^{\chi}(z)=(1-qz^{-1})(1-qz)z^{g-1}\gamma_{\chi}^{-1/2}
\zeta^{\chi}(z^{-1}),\cr
g=\sum_{\chi_1, \chi_2} 
\delta_{\chi_1,\chi_2} g^{\chi_1 \overline{\chi_2}}(x_1/x_2) \in \D_2.
\endgathered$$ 
For any pair $(i,j) \subset \{1, \ldots, r\}$ we set also 
$$g_{i,j}=\sum_{\chi_1, \ldots, \chi_r} 
\delta_{\chi_1, \ldots, \chi_r} g^{\chi_i\overline{\chi_j}}(x_i/x_j).$$
Finally, we put $g(x_1, \ldots, x_r)=\prod_{i<j} g_{i,j} \in \D_r$. 
We are now ready to define our shuffle algebra. Consider the weighted 
symmetrization operator
\begin{align*}
\Psi_r \;: \D_r \to \D_r,\quad
f(x_1, \ldots, x_r) \mapsto \sum_{\sigma \in \mathfrak{S}_r} \sigma 
(f(x_1, \ldots, x_r) g(x_1, \ldots, x_r)).
\end{align*}
Let $\mathbf{A}_r \subset \D_r$ be the image of $\Psi_r$ and set 
$\mathbf{A}^{pr}=\bigoplus_{r \geqslant 0} \mathbf{A}_r$. As in the context of 
Section~1.9 there exists a unique associative algebra structure on 
$\mathbf{A}^{pr}$ fitting in the commutative diagram
\begin{equation}\label{E:shufflediagram2}
\xymatrix{ \D_r\otimes \D_s
\ar[r]^-{\Psi_r \otimes \Psi_{s}} \ar[d]_-{i_{r,s}} & \mathbf{A}_r \otimes 
\mathbf{A}_{s} \ar[d]_-{m_{r,s}} \\
\D_{r+s} \ar[r]^-{\Psi_{r+s}} &
\mathbf{A}_{r+s}.}
\end{equation}

\vspace{.1in}

\begin{prop}\label{P:prinshuffle} The assignement $\mathbf{1}^{\chi}_{1,l} \mapsto \delta_{\chi}x^l \in \mathbf{A}_{1}$ extends to an
algebra isomorphism 
$\Upsilon^{pr}_{\E}: \H^{>,pr}_{\E} \stackrel{\sim}{\to}\mathbf{A}^{pr}$.
\end{prop}

\vspace{.1in}

\noindent
Note that unlike for the spherical Hall algebra $\UU_X$, there are no rational forms $\mathbf{H}^{pr}_R$ of $\mathbf{H}^{pr}_X$~: the principal Hall algebra 
depends on more than just the set $\{\a_1, \ldots, \bar\a_g\}$ of Weil numbers 
of $X$; it depends on the group stucture of $Pic^0(X)$ as well.

\vspace{.2in}

\section{Generalization to arbitrary reductive groups}

\vspace{.15in}

The isomorphism $\Theta$, when restricted to a given rank $r$, 
provides a Langlands 
isomorphism between a subspace $\Vb_r$ 
of the equivariant K-theory $K^{T_s \times GL_r}(C_{\mathfrak{gl}_r})$ and 
the spherical component $\UU^>_X[r]$ of $\H^>_X[r]$, which can be regarded as the space of functions on 
${Bun}_rX={Bun}_{GL_r}X$
which are induced from locally constant function on ${Bun}_{H}X$. 
Here $H \subset GL_r$ is a maximal torus.
This isomorphism may be generalized to an arbitrary algebraic reductive group $G$ in place of $GL_r$. 
Indeed, the right hand side of $\Theta$, i.e.,
the spherical component of $Fun({Bun}_GX)$, may be defined in 
general and described using the Gindikin-Karpelevich formula, while the 
left hand side makes sense for an arbitrary 
reductive group $G^{\vee}$ as well and may be computed using the method of Section~2. We briefly explain 
this in the present section.

\vspace{.2in}

\paragraph{\textbf{B.1.}} Let $G$ be a reductive algebraic group defined over $\mathbb{F}_q$. 
Let $F=\mathbb{F}_q(X)$ be the function field of $X$. For any closed point $x \in X$ 
we denote by $\widehat{O}_x$ the completed local ring at $x$ and by $F_x$ is field of fractions. 
Let $\mathbb{A}_X$, resp. $\mathbb{O}_X$ be the ring of ad\`eles, resp. integer ad\`eles of $X$.
Following Weil, we identify the set $Bun_GX$ of $G$-bundles over $X$ with a double quotient
$$Bun_GX \simeq G(F) \backslash G(\mathbb{A}_X) /G(\mathbb{O}_X).$$ 
Let $B \subset G$ be a Borel subgroup and let $H=B/U$ be the Levi factor (a torus). The
sets of $H$-bundles and $B$-bundles over $X$ are
$$Bun_BX \simeq B(F) \backslash B(\mathbb{A}_X) /B(\mathbb{O}_X), \qquad 
Bun_HX \simeq H(F) \backslash H(\mathbb{A}_X) /H(\mathbb{O}_X).$$ 
Let  $AF_{X,G}$ be the set of compactly supported functions $f: Bun_GX \to \CC$. 
We define $AF_{X,H}$ in a similar way.
The natural inclusion and projection $B \to G$, $B \to H$ induce maps
$$\xymatrix{ Bun_HX & Bun_BX \ar[r]^-{p_2} \ar[l]_-{p_1} & Bun_GX.}$$
The map $p_1$ is smooth. We define the (Eisenstein) induction map as 
\begin{equation*}
\begin{split}
Ind_H^G : AF_{X,H} \to AF_{X,G},\quad
 f \mapsto p_{1!}p_2^*(f),
 \end{split}
 \end{equation*}
and the constant term map as 
\begin{equation*}
\begin{split}
Res_G^H: AF_{X,G} \to Fun(Bun_HX, \CC),\quad
g \mapsto p_{2!}p_1^*(g).
\end{split}
\end{equation*}
For $G=GL_n$ we have $AF_{X,G}=\H_X^>[n]$, $AF_{X,H}=(\H_X^>[1])^{\otimes n}$ and the 
induction/restriction maps correspond to the product/coproduct in the Hall algebra of $X$.

\vspace{.2in}

\paragraph{\textbf{B.2.}}
The connected components of $Bun_HX$ are parametrized by elements of the 
lattice of cocharacters 
$P^\vee= Hom(\CC^\times, H)$. To $\lambda^\vee \in P^\vee$ we associate the 
connected component 
of the $H$-bundle $\mathcal{F}_{\lambda^\vee}$ 
induced from some line bundle $\mathcal{F}\in Pic^1(X)$ by the
map $\lambda^{\vee} : \CC^\times \to H$.  
Let $1_{\lambda^\vee} \in AF_{X,H}$ 
be the characteristic function of
the connected component of $\mathcal{F}_{\lambda^\vee}$. Let 
$$\UU_{X,H} \subset AF_{X,H}$$ be the space of locally constant functions. 
It is identified with the
group algebra $\CC[P^\vee]$ in the obvious way. 
Let $\Delta^\vee \subset P^\vee$ 
be the coroot system of $G$, and let 
$\Delta^{\vee,+}$ be the set of positive coroots corresponding to $B$. 
We likewise define $P^{\vee,+}$ 
to be the positive cone in $P^\vee$. 
Let $\widehat{\UU}_{X,H}=\widehat{\CC[P^\vee]}$ be the formal completion
of $\UU_{X,H}=\CC[P^\vee]$ in the direction of $P^{\vee,+}$. An element of ${\widehat{\UU}}_{X,H}$ is a formal 
linear combination  $\sum_{\lambda^\vee} a_{\lambda^\vee} 1_{\lambda^\vee}$ whose support lies in a finite 
union of translates of $P^{\vee,+}$. For $s(z) \in \CC(z)$ a rational function and $\mu^\vee \in P^\vee$ we 
denote by $s(1_{\mu^\vee}) \in \widehat{\UU}_{X,H}$ the expansion in the direction of $P^{\vee,+}$. Here we 
use the multiplicative structure
of $\CC[P^\vee]$, i.e., we have 
$$1_{\lambda^\vee} 1_{\mu^\vee}=1_{\lambda^\vee+\mu^\vee}.$$
Finally, we set $u_{\lambda^\vee}=Ind_H^G ( 1_{\lambda^\vee})$, and we let 
$$\UU_{X,G} \subset AF_{X,G}$$ denote the subspace linearly spanned by the  set
$\{u_{\lambda^\vee}; \lambda^\vee \in P^\vee\}$. 
For $G=GL_n$ we have $\UU_{X,G} = \UU^>_X[n]$ 
with $\UU^>_X$ being the spherical Hall algebra of $X$.

\vspace{.1in}

\begin{prop}\label{P:GKarp} The following holds.

(a) the map $Res_G^H : \UU_{X,G} \to Fun(Bun_HX,\mathbb{C})$ 
is injective and takes values in 
$\widehat{\UU}_{X,H}$,

(b) for any $\lambda^\vee \in P^\vee$ we have
$$Res_G^H(u_{\lambda^\vee})=Res_G^H \circ Ind_H^G (1_{\lambda^\vee})
=\sum_{\sigma \in W} c_{\sigma} 1_{w \lambda^\vee}$$
where 
$$c_{\sigma}=
\prod_{\alpha \in \Psi_{\sigma}} q^{1-g} 
\frac{\zeta_X(1_{-\alpha^\vee})}{\zeta_X(q1_{-\alpha^\vee})},\quad
\Psi_{\sigma}=\Delta^+ \cap w^{-1}(\Delta^+).$$
\end{prop}

\noindent Statement $(a)$ is proved as in the $GL_n$ case, 
using the nondegeneracy 
of the natural pairing in $\UU_{X,H}$ 
and the adjunction of $Res_G^H$ and $Ind_H^G$. 
Statement $(b)$ is the Gindikin-Karpelevich 
formula, see e.g. \cite{GindiKarp} and \cite{Casselmann}.
As in the $GL_n$ case, it is useful to rephrase the above proposition in combinatorial terms\footnote{Note that 
the sign convention here is different from the one used for $GL_n$, 
see Remark 3.2.2}. 
Put $e_X(z)=z^{1-g}\tilde{\zeta}_X(z)$, where $\tilde{\zeta}_X(z)$ is given by (\ref{E:deftzetax(z)}). 
Define a symmetrization operator 
$$\Psi_G : \CC[P^\vee] \to \CC[P^\vee]^W,\quad
\Psi_G (1_{\lambda^\vee}) =
\sum_{\sigma \in W} \sigma \big\{ \big(\prod_{\alpha \in \Delta^+} 
e_X(1_{\alpha^\vee}) \big)\cdot 1_{\lambda^\vee}\big\}.$$
We will denote the image of $\Psi_G$ by $\mathbf{A}_G$. 
Consider the map 
$$\Xi : \CC[P^\vee]^W \to \widehat{\CC[P^\vee]}, \quad
1_{\lambda^\vee} \mapsto \big(\prod_{\alpha \in \Delta^+} 
e_X^{-1}(1_{\alpha^\vee})\big) 1_{\lambda^\vee}.$$
From Proposition~\ref{P:GKarp} $(b)$ we 
see that the following diagram is commutative
\begin{equation}\label{D:diagGKarp1}
\xymatrix{ \CC[P^\vee] \ar[rr]^-{\Psi_G} \ar[d]^-{\wr} & &\CC[P^\vee]^W \ar[d]^-{\Xi} &\\
\UU_{X,H} \ar[rr]^-{Res_G^H \circ Ind_H^G} && \widehat{\CC[P^\vee]} \ar[r]^-{\sim} & \widehat{\UU}_{X,H}.}
\end{equation}
The leftmost vertical map and the rightmost horizontal one are the canonical isomorphisms. 
We deduce that there exists a (unique)
isomorphism $\iota_G : \UU_{X,G} \to \mathbf{A}_G$ making the next diagram commutative
\begin{equation}\label{D:diagGKarp2}
\xymatrix{ \CC[P^\vee] \ar[rr]^-{\Psi_G} \ar[d]^-{\wr} & &\mathbf{A}_G \ar[d]^-{\Xi} \\
\UU_{X,H} \ar[r]^-{Ind_H^G} & \UU_{X,G} \ar[r]^-{Res_G^H} \ar[ur]^-{\iota_G} & \widehat{\UU}_{X,H}}
\end{equation}
As in the $GL_n$ case, we twist the induction product. 
Let $\rho$ be the 
half-sum of all positive roots of $G$. We define a new induction product 
$\dot{Ind}_H^G: \UU_{X,H} \to \UU_{X,G}$ by
$$ \dot{Ind}_H^G(1_{\lambda^\vee})=
Ind_H^G(1_{2(g-1) \rho^\vee +\lambda^\vee}).$$
Next, define a new symmetrization operator $\dot{\Psi}_G: \CC[P^\vee] \to \CC[P^\vee]^{W}$ by
$$\dot{\Psi}_G (1_{\lambda^\vee}) =\sum_{\sigma \in W} \sigma  \big\{ \big(\prod_{\alpha \in \Delta^+} \tilde{\zeta}_X(1_{\alpha^\vee}) \big)\cdot 1_{\lambda^\vee}\big\}.$$
We will denote the image of $\dot{\Psi}_G$ by $\dot{\mathbf{A}}_G$. 
It is easy to check that 
there is a commutative diagram
\begin{equation}\label{D:diagGKarp3}
\xymatrix{ \CC[P^\vee] \ar[rr]^-{\dot{\Psi}_G} \ar[d]^-{\wr} & &\dot{\mathbf{A}}_G \ar[d]^-{\Xi} \\
\UU_{X,H} \ar[r]^-{\dot{Ind}_H^G} & \UU_{X,G} \ar[r]^-{Res_G^H} \ar[ur]^-{\dot{\iota}_G} & \widehat{\UU}_{X,H}}
\end{equation}
for some (unique) isomorphism 
$\dot{\iota}_G : \UU_{X,G} \to \dot{\mathbf{A}}_G$.
We summarize all this in the following fashion.

\vspace{.1in}

\begin{prop} There is a canonical isomorphism 
$\dot{\iota}_G : \UU_{X,G} \to \dot{\mathbf{A}}_G$ making the 
following diagram commutative 
\begin{equation}\label{D:diagGKarp4}
\xymatrix{ \CC[P^\vee] \ar[rr]^-{\dot{\Psi}_G}  & &\dot{\mathbf{A}}_G \\
\UU_{X,H} \ar[u]_-{\wr} \ar[rr]^-{\dot{Ind}_H^G} & & \UU_{X,G}  \ar[u]^-{\dot{\iota}_G} }
\end{equation}
\end{prop}

\vspace{.1in}

\noindent This proposition has a useful corollary. 
Observe that $\CC[P^\vee], \CC[P^\vee]^W$ do not depend on 
$X$, while $\dot{\Psi}_G$ only depends (polynomially) on the 
Weyl numbers $\a_1, \bar\a_1, \ldots, \a_g, \bar\a_g$. 
Hence all of these may be defined over the representation ring $R_a$ of the torus $T_a$. It follows that there 
exists 'universal' versions $\UU_{R_a,H}, \UU_{R_a,G}$ and $\dot{Ind}_H^G$ of 
$\UU_{X,H}, \UU_{X,G}$ and $\dot{Ind}_H^G$ defined over $R_a$, 
depending only on the genus $g$ of $X$,
and which specialize to $\UU_{X,H}, \UU_{X,G}$ and $\dot{Ind}_H^G$ for any 
particular choice of $X$. Of 
course, (\ref{D:diagGKarp4}) holds for these $R_a$-forms, i.e.,
there is a commutative diagram
\begin{equation}\label{D:diagGKarp45}
\xymatrix{ {R_a}[P^\vee] \ar[rr]^-{\dot{\Psi}_G}  & &\dot{\mathbf{A}}_{R_a,G} \\
\UU_{R_a,H}\ar[u]_-{\wr} \ar[rr]^-{\dot{Ind}_H^G} & & \UU_{R_a,G}  \ar[u]^-{\dot{\iota}_G} }
\end{equation}

\vspace{.2in}

\paragraph{\textbf{B.3.}} Let $G^{\vee}$ denote the Langlands dual group to $G$ (a \textit{complex} reductive 
group), and let $H^{\vee}$ be a maximal torus of $G^{\vee}$. The lattice of characters $Hom(H^\vee,\CC^\times)$ of 
$H^{\vee}$ is identified with $P^\vee$. Let $\mathfrak{g}^{\vee}$ and $\mathfrak{h}^{\vee}$ be the respective 
Lie algebras of $G^\vee$ and $H^\vee$. Put
$$C_{\mathfrak{g}^{\vee}}=
\{(a,b) \in (\mathfrak{g}^{\vee})^g\times(\mathfrak{g}^{\vee})^g\;;\;
\sum_r ad(a_r)(b_r)=0\},\quad 
C_{\mathfrak{h}^{\vee}}=\mathfrak{h}^{\vee})^g \times (\mathfrak{h}^{\vee})^g.$$
The torus $T_s$ acts on both $C_{\mathfrak{g}^{\vee}}$ and 
$C_{\mathfrak{h}^{\vee}}$, see Section~2.2. We 
may apply the method of Section~2.3 \textit{verbatim} to define a
convolution operation 
$$\mu~:K^{T_s \times H^{\vee}}(C_{\mathfrak{h}^{\vee}}) \to 
K^{T_s \times G^{\vee}}(C_{\mathfrak{g}^{\vee}}).$$
We set 
$$\Vb_{H^\vee}=K^{T_s \times H^{\vee}}(C_{\mathfrak{h}^{\vee}}), 
\quad
\Vb_{G^\vee}=
K^{T_s \times G^{\vee}}(C_{\mathfrak{g}^{\vee}}).$$
Let $R_{G^{\vee}}$, $R_{H^{\vee}}$ denote the complexified representation 
rings of $T_s \times G^{\vee}$, $T_s \times H^{\vee}$. We have
\begin{equation}\label{E:identification1}
R_{G^{\vee}} = R_s[P^{\vee}]^{W}, \qquad R_{H^{\vee}}=R_s[P^{\vee}]
\end{equation}
where $P^{\vee}=Hom(H^{\vee},\CC^\times)$ is the character group of $H^{\vee}$ 
and $W$ is the Weyl group of $(G^{\vee},H^\vee)$. The restriction to $\{0\}$ 
gives rise, as in Section~2.5, to maps
$$\gamma_{G^{\vee}} : \Vb_{G^{\vee}} 
\to R_{G^{\vee}}, \qquad 
\gamma_{H^{\vee}} : \Vb_{H^{\vee}} 
\to R_{H^{\vee}}.$$ 
The map $\gamma_{H^{\vee}}$ is an isomorphism. We conjecture that 
$\gamma_{G^{\vee}}$ is injective, i.e.,
that $\Vb_{G^{\vee}}$ is a torsion free 
$R_{G^{\vee}}$-module. 
We denote by $\bar\Cb_{G^{\vee}}$ the quotient of 
$\mu(\Vb_{H^\vee})$ by its torsion submodule.
The proof of the following result is identical to that of 
Proposition~\ref{P:KHall}.

\vspace{.1in}

\begin{prop}\label{P:KHall2} There is a commutative diagram
$$\xymatrix{
\Vb_{H^\vee}\ar[r]^{\mu}\ar[d]_{\gamma_{H^\vee}}^-{\wr}&
\Vb_{G^\vee}\ar[d]^{\gamma_{G^\vee}}\cr
R_{H^\vee}\ar[r]^{\nu}&R_{G^\vee},}
$$
where the map $\nu$ is given by
\begin{equation}\label{E:KHallform2}
\nu(e^{\lambda^\vee})=\sum_{ w \in W} 
w\big( k_{\Delta^\vee}\;e^{\lambda^\vee}\big)
\end{equation}
where
\begin{equation}\label{E:Kg(z)2}
\begin{gathered}
k_{\Delta^\vee}=\prod_{\alpha \in \Delta^{\vee+}} k(e^{\alpha}), \cr k(z)=
(1-z)^{-1}(1-p^{-1}z^{-1})\prod_r(1-x_r^{-1}z)(1-y_r^{-1}z),
\end{gathered}
\end{equation}
and where $\Delta^{\vee +}$ is the set of positive roots of $\mathfrak{g}^\vee$.
\end{prop}

\vspace{.1in}

\noindent
Set $\mathbf{B}_{G^\vee}=\nu(R_{H^\vee})$.
By Proposition \ref{P:KHall2}, there is a unique isomorphism 
$j_{G^\vee}: \bar\Cb_{G^\vee} \to 
\mathbf{B}_{G^\vee}$ fitting in a commutative diagram
\begin{equation}\label{D:diagGKarp5}
\xymatrix{{\Vb}_{H^\vee} \ar[rr]^-{\mu}  
\ar[d]^-{\wr}& &\bar\Cb_{G^\vee}\ar[d]^-{j_{G^\vee}} \\
R_{H^\vee} \ar[rr]^-{\nu} & & \mathbf{B}_{G^\vee}. }
\end{equation}

\vspace{.2in}

\paragraph{\textbf{B.4.}} Identify $T_a$ with $T_s$ and $R_a$ with 
$R_s$ as in Section 3.1. Using (\ref{E:identification1}) we get an 
identification between the maps 
$\dot{\Psi}_{G}: R[P^\vee] \to \dot{\mathbf{A}}_{R,G}$ and 
$\nu: R_{H^\vee} \to \mathbf{B}_{G^\vee}$. Combining (\ref{D:diagGKarp45}) 
and (\ref{D:diagGKarp5}), we immediately get

\vspace{.1in}

\begin{theo} There is a canonical isomorphism 
$\Theta_{R,G} : \bar\Cb_{G^\vee}{\to}  \UU_{R,G}$ 
making the following diagram commutative
\begin{equation}\label{D:diagGKarp6}
\xymatrix{{\Vb}_{H^\vee} \ar[rr]^-{\mu}  \ar[d]^-{\wr}& &
\bar\Cb_{G^\vee}\ar[d]^-{\Theta_{R,G}} \\
\UU_{R,H} \ar[rr]^-{\dot{Ind}_H^G} & & \UU_{R,G}. }
\end{equation}
\end{theo}

\vspace{.2in}

\paragraph{\textbf{B.5.}} Fix a point $x \in X(\fq)$. 
There is an action $h_x$ of the representation ring $Rep_{G^\vee}$ on 
$AF_{X,G}$ by means of the Hecke operators at $x$,
see e.g., \cite[Section~2.3]{Frenkel}. There is a similar action of 
$Rep_{H^\vee}$ on $AF_{X,H}$, and we have 
\begin{equation}\label{E:compathecke}
\begin{split}
L \cdot Ind_H^G(f)=Ind(Res_{G^\vee}^{H^\vee} L \cdot f), \quad 
L \in Rep_{G^\vee}, \quad f \in AF_{X,H},\\
Res_{H^\vee}^{G^\vee}(L \cdot g)=Res_{G^\vee}^{H^\vee}(L) \cdot 
Res_{G^\vee}^{H^\vee}(g), \quad L \in Rep_{G^\vee}, \quad g \in AF_{X,G}.
\end{split}
\end{equation}
For $\lambda^\vee \in P^\vee$ let $e^{\lambda^\vee} \in Rep_{H^\vee}$ 
the class of the corresponding one-dimensional module. A direct computation 
shows that 
\begin{equation}\label{E:compathecke2}
e^{\lambda^\vee} \cdot 1_{\mu^{\vee}}=1_{\lambda^\vee + \mu^\vee}
=1_{\lambda^\vee} 1_{\mu^\vee},\quad\lambda^\vee,\mu^\vee\in P^\vee.
\end{equation}
The product on the left hand side is the Hecke action while the product 
on the right hand side is the multiplication in the group ring $\CC[P^\vee]$. 
Note that the Hecke action $h_x$ on $\UU_{X,H}$ and hence on $\UU_{X,G}$ 
is independent of the choice of the point $x$.
On the other hand, there is a natural action $\rho$ of the representation ring 
$R_{G^\vee}$ on $\bar\Cb_{G^\vee}$ by tensor product. 

\begin{theo} The isomorphism $\Theta_{R,G}$ intertwines the above actions of 
$Rep_{G^\vee}$ on $\bar\Cb_{G^\vee}$ and $\UU_{R,G}$ respectively, 
i.e., we have $\Theta_{R,G} \circ \rho(L)  = h_x(L) \circ \Theta_{R,G}$ for any 
$L \in Rep_{G^\vee}$.
\end{theo}

\begin{proof}
The restriction map $\gamma_{G^\vee} : \Vb_{G^\vee}\to R_{G^\vee}$ 
is $R_{G^\vee}$-linear. Likewise, the map
$\gamma_{H^\vee} : \Vb_{H^\vee}\to R_{H^\vee}$ is $R_{H^\vee}$-linear. 
Letting $R_{G^\vee}$ act on 
$\Vb_{H^\vee}$ and 
$R_{H^\vee}$ by means of the embedding $R_{G^\vee} \subset R_{H^\vee}$, we see
from (\ref{E:KHallform2}) that $\mu, \nu$ are $R_{G^\vee}$-linear as well, 
and hence that $j_{G^\vee}$ is $R_{G^\vee}$-linear.
By (\ref{E:compathecke2}) the identification 
$\UU_{R,H} \simeq R[P^\vee]$ is $R_{H^\vee}$-linear. 
Let $R_{G^\vee}$ act on $R[P^\vee]$ and $\UU_{R,H}$ by means of the embedding 
$R_{G^\vee} \subset R_{G^\vee}$. By construction, the map $\dot{\Psi}_G$ is 
$R_{G^\vee}$-linear. By (\ref{E:compathecke}), the map $\dot{Ind}^G_H$ is also 
$R_{G^\vee}$-linear. But then $\dot{\iota}_G$ is $R_{G^\vee}$-linear as well. 
We have shown that $j_{G^\vee}$ and $\dot{\iota}_G$ are $R_{G^\vee}$-linear. 
The identification $\dot{\mathbf{A}}_G \simeq \mathbf{B}_{G^\vee}$ is 
obviously $R_{G^\vee}$-linear. The theorem follows.
\end{proof}

\vspace{.2in}

\paragraph{\textbf{B.6.}} To finish we describe, as in Section~3.3, 
the image under the isomorphism 
$\Theta_{K,G}$ of the skyscraper sheaf
$[\mathcal{O}_{\{0\}}] \in 
\Vb_{G^\vee}\otimes_R K$. 
The proof is the same as in 
the case of $\mathfrak{g}=\mathfrak{gl}_r$, see Proposition~\ref{P:cohtriv}. 
Let us denote by $\mathcal{B}=G/B$ the flag variety of $G$.

\begin{prop}\label{P:cohtrivgeneral} We have
\begin{equation}\label{E:cohtrivG}
\Theta_{K,G}([\mathcal{O}_{\{0\}}])=
(-1)^{dim\;\mathcal{B}}\frac{q^{-g\;dim \mathcal{B}}}{\#\mathcal{B}(\fq)} \dot{Ind}_H^G \bigg(1_{-2\rho^\vee} \prod_{\substack{\sigma \in \Delta^+ }}
\prod_{l=1}^g (1-\a_l 1_{\sigma^\vee})(1-\bar\a_l1_{\sigma^\vee})\bigg)
\end{equation}
where $\Delta^+$ is the set of positive roots of $(G,B)$.
\end{prop}

\vspace{.2in}

\small{}

\vspace{4mm}

\noindent
O. Schiffmann, \texttt{olive@math.jussieu.fr},\\
D\'epartement de Math\'ematiques, Universit\'e de Paris 6, 175 rue du Chevaleret, 75013 Paris, FRANCE.

\vspace{.1in}

\noindent
E. Vasserot, \texttt{vasserot@math.jussieu.fr},\\
D\'epartement de Math\'ematiques, Universit\'e de Paris 7, 175 rue du Chevaleret, 75013 Paris, FRANCE.

\end{document}